\documentclass[12pt, reqno]{amsart}

\date{July 3, 2019}

\usepackage[parfill]{parskip}    
\usepackage{graphicx}
\usepackage{amssymb}
\usepackage{amsmath}
\usepackage{amsthm}
\usepackage{mathabx}
\usepackage{commath}
\usepackage{epstopdf}

\usepackage[T1]{fontenc}
\usepackage[utf8]{inputenc}

\usepackage[ngerman, german, english]{babel} 
\usepackage{bbm}

\newcommand{\R}{\mathbb{R}} 
\newcommand{\Z}{\mathbb{Z}} 
\newcommand{\N}{\mathbb{N}} 
\newcommand*\diff{\mathop{}\!\mathrm{d}} 
\newcommand{\SL}{\Sigma_\lambda}

\newtheorem{theorem}{Theorem}

\newtheorem{corollary}[theorem]{Corollary}
\newtheorem{lemma}[theorem]{Lemma}

\newtheorem{remarknumbered}[theorem]{Remark}

\newtheorem*{definition*}{Definition}

\begin{document}

\title[Singular solutions to a critical biharmonic equation]{Singular solutions to a semilinear biharmonic equation with a general critical nonlinearity}

\author{Rupert L. Frank}
\address[Rupert L. Frank]{Mathematisches Institut, Ludwig-Maximilans Universit\"at M\"unchen, Theresienstr. 39, 80333 M\"unchen, Germany, and Mathematics 253-37, Caltech, Pasa\-de\-na, CA 91125, USA}
\email{r.frank@lmu.de, rlfrank@caltech.edu}

\author{Tobias König}
\address[Tobias König]{Mathematisches Institut, Ludwig-Maximilans Universit\"at M\"unchen, Theresienstr. 39, 80333 M\"unchen, Germany}
\email{tkoenig@math.lmu.de}

\dedicatory{To V. Maz'ya on the occasion of his 80th birthday}

\begin{abstract}
We consider positive solutions $u$ of the semilinear biharmonic equation $\Delta^2 u = |x|^{-\frac{n+4}{2}} g(|x|^\frac{n-4}{2} u)$ in $\R^n \setminus \{0\}$ with non-removable singularities at the origin. Under natural assumptions on the nonlinearity $g$, we show that $|x|^\frac{n-4}{2} u$ is a periodic function of $\ln |x|$ and we classify all such solutions.
\end{abstract}

\thanks{\copyright\, 2019 by the authors. This paper may be
reproduced, in its entirety, for non-commercial purposes.\\
Partial support through US National Science Foundation grant DMS-1363432 (R.L.F.) and Studienstiftung des deutschen Volkes (T.K.) is acknowledged.}

\maketitle

\section{Introduction and main results}

We are interested in positive solutions $u$ of the semilinear biharmonic equation 
\begin{equation}
\label{differential eq}
\Delta^2 u = \frac{1}{|x|^{\frac{n+4}{2}}}\, g(|x|^\frac{n-4}{2} u)  \qquad\text{in}\ \R^n \setminus \{0\}
\end{equation}
which may have singularities at the origin. We always assume $n\geq 5$. The nonlinearity $g \in C^1(\R_+)$ is taken to satisfy a set of growth conditions which will be specified below. The seemingly strange way the right hand side of equation \eqref{differential eq} is written ensures invariance of the equation under rotations, dilations and inversion in the unit sphere. The form becomes very natural later when we pass to logarithmic radial coordinates. We will always interpret \eqref{differential eq} in the weak sense, that is, we assume $u \in H^2_{\rm loc}(\R^n \setminus \{0\})$, $g(u)\in L^1_{\rm loc}(\R^n\setminus\{0\})$ and 
\[ \int_{\R^n} \Delta u\;  \Delta \varphi =   \int_{\R^n} g(u) \varphi  \qquad \text{ for all } \varphi \in C^\infty_0 (\R^n \setminus \{0\}) \,. \]

Our goal is to classify all positive solutions of \eqref{differential eq}. Under some natural assumptions on $g$ we will be able to show that, if $u$ has a non-removable singularity at $0$, then $|x|^\frac{n-4}{2} u$ is a periodic function of $\ln |x|$ and, up to dilations, all such functions are uniquely parametrized by the minimal value of this periodic function.

The simplest example of a nonlinearity to which our results apply is $g(t)=t^\frac{n+4}{n-4}$, where the equation becomes
\begin{equation}
\label{differential eq model}
\Delta^2 u = u^\frac{n+4}{n-4}  \qquad\text{in}\ \R^n \setminus \{0\} \,.
\end{equation}
Note that the exponent in this equation is critical in the sense of Sobolev's embedding theorem. We have treated equation \eqref{differential eq model} in our previous paper \cite{FrKo} and our goal now is to extend these results to a more general class of equations. In addition, we will be able to simplify some parts of the argument in \cite{FrKo}.

Other equations that we can treat in this paper are, for instance,
\begin{equation}
\label{differential eq model2}
\Delta^2 u = |x|^{-\frac{n+4}{2}+q\frac{n-4}{2}}\, u^q  \qquad\text{in}\ \R^n \setminus \{0\} \,.
\end{equation}
with $1<q\leq \frac{n+4}{n-4}$, as well as equations whose right hand side is equal to a finite sum, with positive coefficients, of terms as on the right side of \eqref{differential eq model2} with different $q$'s. We can also allow for a `Hardy--Rellich' term on the right side, that is,
\begin{equation}
\label{differential eq model3}
\Delta^2 u = \beta |x|^{-4} u + |x|^{-\frac{n+4}{2}+q\frac{n-4}{2}}\, u^q  \qquad\text{in}\ \R^n \setminus \{0\}
\end{equation}
with $1<q\leq\frac{n+4}{n-4}$, provided the constant satisfies $0<\beta<\frac{n^2(n-4)^2}{16}$. Note that $\frac{n^2(n-4)^2}{16}$ is the sharp constant in the Hardy--Rellich inequality. 

The precise assumptions on the nonlinearity $g$ are as follows,
 \begin{equation}
\label{conditions on g}
\begin{cases}
g \in C^1(\R_+), \; g > 0, \; \lim_{t \to 0} g(t) = 0, \\
\frac{g(t)}{t} < g'(t)  \leq \frac{n+4}{n-4} \frac{g(t)}{t} & \text{ for all } t > 0, \\ 
\beta := \lim_{t \to 0} g'(t) < \frac{n^2(n-4)^2}{16}, \\
g(t) \geq c t^q & \text{ for all } t \geq 1 ,  \quad \text{ for some } q > 1,\, c > 0.
\end{cases}  
\end{equation}
It is part of the assumption on $g$ that the two limits appearing in \eqref{conditions on g} exist. 

Note that equations \eqref{differential eq model2} and \eqref{differential eq model3} correspond to the choices $g(t)=t^q$ and $g(t) = \beta t + t^q$, respectively. These functions clearly satisfy all the requirements in \eqref{conditions on g}.

Let us discuss assumptions \eqref{conditions on g} for general functions $g$. The inequality $g'(t)  \leq \frac{n+4}{n-4} \frac{g(t)}{t}$ means that the non-linearity is subcritical or critical in the sense of Sobolev's embedding theorem, while the inequality $g(t) \geq c t^q$ is a superlinearity assumption for large values of the argument of $g$.

We proceed in two steps, the first one being a proof that positive solutions of \eqref{differential eq} are radial and the second one being a careful analysis of the resulting ordinary differential equation.

The following is our first main result.

\begin{theorem}[Radial symmetry]
\label{theorem hardy solutions are radial}
Suppose that $g$ satisfies \eqref{conditions on g} and let $u \in H^2_{\rm loc}(\R^n \setminus \{0\})$ be a solution of \eqref{differential eq} with $u >0$. 

Assume that either $u\not\in L^\frac{2n}{n-4}(\R^n)$ or that the inequality $g'(t)  \leq \frac{n+4}{n-4} \frac{g(t)}{t}$ in \eqref{conditions on g} is strict for a.e. $t>0$.

Then $u$ is radially symmetric-decreasing with respect to $0$.
\end{theorem}

When $u\in L^\frac{2n}{n-4}(\R^n)$ and the inequality $g'(t)  \leq \frac{n+4}{n-4} \frac{g(t)}{t}$  in \eqref{conditions on g} is not strict for a.e. $t>0$, then the conclusion of the theorem need not hold. Indeed, it is well-known that equation \eqref{differential eq model}, which is translation invariant, has a solution which is strictly radially symmetric-decreasing with respect to an \emph{arbitrary} given point. (This follows from the existence of an optimizer in the Sobolev inequality for the Bilaplacian, proved in dual form in \cite{Lieb}; for a uniqueness results for the corresponding Euler--Lagrange equation see also \cite{Swanson}.) This shows that some extra condition is needed to conclude radial symmetry with respect to the origin, although the condition given in the theorem can probably be relaxed.

\begin{remarknumbered}\label{remark strict}
We will show in Section \ref{section ode} that, in fact, $\frac{\partial u}{\partial |x|}<0$ for all $x\in\R^n\setminus\{0\}$, using ODE methods. 
\end{remarknumbered}

We now proceed to the second step of our argument. According to Theorem \ref{theorem classification ODE} we can write equation \eqref{differential eq} as an ordinary differential equation. It becomes particularly simple in logarithmic coordinates. That is, we make the so-called Emden--Fowler change of variables
$$
u(x) = |x|^{-\frac{n-4}{2}} v(\ln |x|) \,.
$$
After a lengthy, but straightforward computation, we can write \eqref{differential eq} as 
\begin{equation}
\label{eq ODE intro}
v^{(4)}-  A v'' + B v = g(v) 
\qquad\text{in}\ \R
\end{equation}
with
$$
A = \frac{n(n-4)+8}{2}
\qquad\text{and}\qquad
B = \frac{n^2(n-4)^2}{16} \,.
$$
A similar change of variables is used in \cite{Lieb} in a dual form. For fourth order ODEs it appears also in \cite{GazzolaGrunau}.

The following is our second main result. Recall from \eqref{conditions on g} that by definition $\beta = \lim_{t \to 0} g'(t)$.

\begin{theorem} [Classification of ODE solutions]
\label{theorem classification ODE}
Suppose that $g$ satisfies \eqref{conditions on g}.

Then any positive solution $v\in C^4(\R)$ of \eqref{eq ODE intro} is either constant, or homoclinic to zero, or periodic. These solutions can be classified, up to translations, by their minimal value in the following sense.
\begin{enumerate}
\item[(i)] There is a unique $a_0 > 0$ such that $g(a_0) = B a_0$. Moreover, if $v \in C^4(\R)$ is a positive solution of \eqref{eq ODE intro}, then $\inf_\R v \leq a_0$, with equality if and only if $v$ is a non-zero constant.
\item[(ii)] If $a \in \left(0, a_0 \right)$, then there is a unique (up to translations) bounded solution $v\in C^4(\R)$ of \eqref{eq ODE intro} with $\inf_\R v = a$. This solution is periodic, has a unique local maximum and minimum per period and is symmetric with respect to its local extrema. 
\item[(iii)] There is a unique (up to translations) positive solution $v \in C^4(\R)$ of \eqref{eq ODE intro} with $\inf_\R v=0$. This solution is symmetric-decreasing and satisfies $v(t)\leq C_\epsilon e^{-(\sqrt\mu-\epsilon)|t|}$ for any $\epsilon>0$, where $\mu = \frac12 (A - \sqrt{A^2 - 4(B-\beta)})$. Moreover, if
\begin{equation}
\label{eq:gnearzero}
|g(t)-\beta t| \leq C t^r \quad\text{for all}\ t\leq 1\,,\qquad \text{for some}\ r>1,\, C<\infty \,,
\end{equation}
then $\lim_{|t|\to\infty} e^{\sqrt\mu |t|} v(t)$ exists and is finite. When $g(t)\geq \beta t$ for all $t>0$, then the limit is positive.  
\end{enumerate}
\end{theorem}

Note that for homoclinic solutions we prove exponential decay, since the assumption $\beta<B$ implies $\mu>0$. Moreover, recalling the explicit expression of $A$ and $B$, we obtain
$$
\sqrt\mu = \frac{n-4}{2}
\qquad\text{if}\ \beta=0 \,.
$$

The combination of Theorems \ref{theorem hardy solutions are radial} and \ref{theorem classification ODE} yields immediately the following classification of positive singular solutions of the PDE \eqref{differential eq}. To state this result, we denote, for $a \in [0, a_0]$, by $v_a$ the unique positive solution to \eqref{eq ODE intro} obtained from Theorem \ref{theorem classification ODE} by requiring that $\inf_\R v_a = a$ and $v_a(0) = \max_\R v$. For $a\in(0,a_0)$ we denote by $L_a$ the minimal period length of $v_a$ and we set $L_{a_0}=0$ and $L_0 = \infty$.

\begin{theorem}[Classification of PDE solutions]
\label{theorem classification}
Suppose that $g$ satisfies \eqref{conditions on g}. Let $u \in C^4(\R^n \setminus \{0\})$ be a positive solution of \eqref{differential eq} and assume that either $u\not\in L^\frac{2n}{n-4}(\R^n)$ or that the inequality $g'(t)  \leq \frac{n+4}{n-4} \frac{g(t)}{t}$ in \eqref{conditions on g} is strict for a.e. $t>0$.

Then there are $a \in [0, a_0]$ and $L \in [0, L_a]$ such that
\[ u(x) = |x|^{-\frac{n-4}{2}} v_a(\log |x| + L) \,, \]
where $v_a$ is the solution of \eqref{eq ODE intro} introduced above.

In particular, if $a=a_0$, then $u(x) = |x|^{-\frac{n-4}{2}} a_0$, and if $a\in(0,a_0)$, then
\begin{align*}
0 & < \liminf_{|x|\to 0} |x|^\frac{n-4}{2} u(x) = \liminf_{|x|\to\infty} |x|^\frac{n-4}{2} u(x)\\
& < \limsup_{|x|\to 0} |x|^\frac{n-4}{2} u(x) = \limsup_{|x|\to\infty} |x|^\frac{n-4}{2} u(x) <\infty \,.
\end{align*}
Under assumption \eqref{eq:gnearzero}, if $a=0$, then
$$
\lim_{|x|\to 0} |x|^{\frac{n-4}{2}-\sqrt\mu} u(x) = \lim_{|x|\to \infty} |x|^{\frac{n-4}{2}+\sqrt\mu} u(x) <\infty
$$
and, if $g(t)\geq \beta t$ for all $t>0$, then the limit is positive.
\end{theorem}

Let us discuss the implications of this theorem to the question of removability of singularities. When $a>0$, the solution $u$ has a non-removable singularity at the origin. When $a=0$, the situation depends on whether the parameter $\beta$ from \eqref{conditions on g} vanishes or not. For $\beta=0$ (that is, $\sqrt\mu=\frac{n-4}{2}$), the solution extends continuously to the origin, while for $\beta>0$ (that is, $\sqrt\mu<\frac{n-4}{2}$), the solution has a power-like singularity at the origin. Note, however, that this singularity is weaker than in the case $a>0$.

It is also remarkable that the behavior near the origin is closely related to the behavior of $u$ at infinity.

In the remainder of this paper we prove Theorem \ref{theorem hardy solutions are radial}, Remark \ref{remark strict} and Theorem \ref{theorem classification ODE}. The first one is proved in Section \ref{section moving planes} using the method of moving planes, while the second and third one are proved in Section \ref{section ode} using ODE techniques.


\section{Method of moving Planes} \label{section moving planes}

Our goal in this section is to prove Theorem \ref{theorem hardy solutions are radial}. We will deduce it from the following theorem, which is our main symmetry result. 
We point out that for the proof of Theorem \ref{theorem moving planes} below, we actually do not need the lower bound $g(t)/t < g'(t)$ from \eqref{conditions on g}. 

Throughout the following, we will fix a point $a\in\R^n\setminus\{0\}$ and consider
\begin{equation}
\label{definition S}
S = \{0,a\} \subset \R^n \,.
\end{equation}
We shall prove

\begin{theorem}
\label{theorem moving planes}
Suppose that $g$ satisfies \eqref{conditions on g}.

Let $k$ be a positive function in $\R^n\setminus S$ which is symmetric-decreasing with respect to a hyperplane $H$ passing through $0$ and $a$. Assume that $k(x) \gtrsim \text{dist}(x, S)^{-1}$ in a neighborhood of $S$, that $k\in L_{\rm loc}^{(n+2)/4}(\R^n)$ and that $k\in L^n(\Omega)$ for every $\Omega$ which is a positive distance away from $S$.

Let $v \in H^2_{\rm loc}(\R^n \setminus S)$ be a weak solution of
\begin{align}\label{eq:symmeq}
\Delta^2 v = k(x)^\frac{n+4}{2} g(k(x)^{-\frac{n-4}{2}} v(x))
\qquad\text{in}\ \R^n\setminus S
\end{align}
and assume that $v>0$ and that $v \in L^\frac{2n}{n-4}(\Omega)$ for every $\Omega$ which is a positive distance away from $S$. 

Assume that either $v\not\in L^\frac{2n}{n-4}(\R^n)$, or that the inequality $g'(t)  \leq \frac{n+4}{n-4} \frac{g(t)}{t}$ in \eqref{conditions on g} is strict for a.e. $t>0$ and $k$ is strictly symmetric-decreasing.

Then $v$ is strictly symmetric-decreasing with respect to $H$. 
\end{theorem}

By saying that a function $f$ is symmetric-decreasing with respect to a hyperplane $H$ through $0$ with normal vector $e$ we mean that
$$
f(y-te) = f( y+te) \geq f(y+se)
\qquad\text{for all}\ 0\leq s\leq t
\ \text{and all}\  y\cdot e=0 \,.
$$
By saying that $f$ is strictly symmetric-decreasing we mean that the inequality is strict for $s<t$.


\subsection{Kelvin transformation and proof of Theorem \ref{theorem hardy solutions are radial}}\label{sec:kelvin}

Although neither $u$ nor $|x|^{-1}$ from Theorem \ref{theorem hardy solutions are radial} satisfy the assumptions of Theorem \ref{theorem moving planes}, we will now show that we can transform equation \eqref{differential eq} to another equation for which these assumptions are satisfied. Indeed, for $z\neq 0$ we denote by $z^* = \frac{z}{|z|^2}$ the inversion of $z$ about the unit sphere and for any function $u$ on $\R^n\setminus\{0\}$, we define its Kelvin-type transform $u^*_z$ with respect to the point $z$ by
\begin{equation}
\label{definition u z *} u_z^* (x) := \frac{1}{|x|^{n-4}} u(\frac{x}{|x|^2} + z), \qquad x \in \R^n \setminus \{0, -z^*\}. 
\end{equation}
If $z = 0$, we denote $u^* := u_0^*$.

The technique of improving the decay properties of a solution to a conformally invariant equation by passing to the Kelvin transform goes back to \cite{CaGiSp} and has been widely used in the context of the method of moving planes. Since our equation \eqref{differential eq} is, in general, not translation invariant, $u_z^*$ with $z\neq 0$ is, in general, no longer a solution of \eqref{differential eq}. It will, however, satisfy a related equation. Indeed, if $u$ solves \eqref{differential eq}, then from the formula
\[ \Delta^2 \varphi^* (x) = \frac{1}{|x|^{n+4}} \Delta^2 \varphi (\frac{x}{|x|^2}), \quad \varphi \in C^\infty_0(\R^n \setminus \{0\}) \] 
and a straightforward calculation using the fact that $|\frac{x}{|x|^2} - \frac{z}{|z|^2}| = \frac{|x-z|}{|x||z|}$, it follows that $u_z^*$ satisfies the equation
\begin{equation}
\label{equation for kelvin transform u*z} 
\Delta^2 u_z^* (x) = k_z(x)^\frac{n+4}{2} g(k_z(x)^{-\frac{n-4}{2}} u_z^*(x))
\text{ in } \R^n\setminus\{0,-z^*\} 
\end{equation}
with
\[ k_z(x) = \frac{|z^*|}{|x| |x + z^*|} \,,
\qquad x\in\R^n\setminus\{0,-z^*\} \,, \]
for $z\neq 0$ and $k_0(x) = |x|^{-1}$ if $z = 0$. This equation is understood in the weak sense, i.e. 
\begin{equation*}
\int_{\R^n} \Delta u_z^* \Delta \varphi = \int_{\R^n}  k_z(x)^\frac{n+4}{2} g(k_z(x)^{-\frac{n-4}{2}} u_z^*(x)) \varphi(x) \diff x \ \quad \text{for all } \varphi \in C^\infty_0(\R^n \setminus \{0,-z^*\}). 
\end{equation*}
Notice that this means, in particular, that if $u$ solves equation \eqref{differential eq}, then so does $u^*$.  

Here are some more properties of $u^*_z$ which we need for our argument.

\begin{lemma}
\label{lemma properties kelvin transform}
Suppose that $u \in H^2_{\rm loc}(\R^n \setminus \{0\})$. 
\begin{enumerate}
\item[(a)] If $z \neq 0$, then $u_z^* \in H^2_{\rm loc}(\R^n \setminus \{0,-z^*\})$ and, in fact, we have 
\[ \int_\Omega |\Delta u_z^*|^2 < \infty \quad \text{ and } \quad \int_\Omega |u_z^*|^\frac{2n}{n-4} < \infty \]
for every $\Omega$ which is a positive distance away from $0$ and $-z^*$. 
\item[(b)] If $z = 0$, then $u^* \in H^2_{\rm loc}(\R^n \setminus \{0\})$.
\end{enumerate}
\end{lemma}

\begin{proof}
Let $\Omega \subset \R^n$ be a positive distance away from 0 and $-z^*$ and let $\Omega^*=\{ x^*:\ x\in\Omega\}$. Then the set $\Omega^* + z$ is bounded and a positive distance away from 0. By a change of variables, we have 
\begin{equation}
\label{kelvin transform Lp norm} \int_\Omega {u^*_z}^\frac{2n}{n-4}(x) \diff x =  \int_{\Omega^*} u^\frac{2n}{n-4}(x + z) \diff x = \int_{\Omega^* + z} u^\frac{2n}{n-4} < \infty 
\end{equation}
since $u  \in L^\frac{2n}{n-4}_{\rm loc}(\R^n \setminus \{0\})$ by Sobolev embedding. 

Let us now turn to the second derivative. As in \cite[Proof of Lemma 3.6]{Xu}, using Kelvin's transformation rule 
\[ \Delta \big(|x|^{2-n} u(\frac{x}{|x|^2})\big) = |x|^{-n-2} (\Delta u)(\frac{x}{|x|^2}), \]
we have by the product rule for Sobolev functions that $\Delta  u_z^*$ exists weakly in $\R^n\setminus\{0,-z^*\}$ and is given by
\begin{align}
\label{Delta kelvin transform}
&  \Delta u_z^*(x) = \Delta \big( |x|^2 |x|^{2-n} u(\frac{x}{|x|^2} + z) \big) \notag \\
 &=  2n |x|^{2-n} u(\frac{x}{|x|^2} + z) + 4 x \cdot \nabla \Big( |x|^{2-n} u(\frac{x}{|x|^2}+z) \Big) + |x|^{-n} \Delta u(\frac{x}{|x|^2}+z) \notag \\
& = -2(n-4) |x|^{2-n} u(\frac{x}{|x|^2} + z) - 4 |x|^{-n} x \cdot \nabla u(\frac{x}{|x|^2}+ z) + |x|^{-n} \Delta u(\frac{x}{|x|^2}+z)   .
\end{align}
We show square integrability of each of the three terms on the right side. First, 
\begin{align*} 
\int_\Omega \big| |x|^{2-n} u(\frac{x}{|x|^2}+z) \big|^2 &= \int_{\Omega^*}   |x|^{-4} |u(x + z)|^2  \lesssim \int_{\widetilde{\Omega^*}} |\Delta u (x+ z)|^2 < \infty
\end{align*}
by the Hardy--Rellich inequality, where we choose some larger $\widetilde{\Omega^*} \supset \overline{\Omega^*}$ which is still a positive distance away from $-z$. This follows by a simple argument using cutoff functions. For the second term of \eqref{Delta kelvin transform}, we have
 \begin{align*}
 \int_\Omega \Big| |x|^{-n} x \cdot \nabla u(\frac{x}{|x|^2}+z) \Big|^2 & \leq \int_\Omega |x|^{2-2n} |\nabla u(\frac{x}{|x|^2}+z)|^2 = \int_{\Omega^*} |x|^{-2} |\nabla u(x+z)|^2 \\
 & \lesssim  \| u \|_{H^2(\widetilde{\Omega^*} + z)}^2 < \infty
 \end{align*}
by Hardy's inequality applied to $\partial_j u$, $j = 1,...,n$, with some set $\widetilde{\Omega^*}$ as above. 
The third term of \eqref{Delta kelvin transform} gives
\begin{align*}
\int_\Omega |x|^{-2n} |\Delta u(\frac{x}{|x|^2}+z)|^2 = \int_{\Omega^*+ z} |\Delta u (x)|^2 < \infty
\end{align*}
because $u \in H^2_{\rm loc}(\R^n \setminus \{0\})$.

The proof of part (b) is similar, but simpler, and we omit it. 
 \end{proof}

We can now deduce Theorem \ref{theorem hardy solutions are radial} from Theorem \ref{theorem moving planes} in a straightforward way.  

\begin{proof}
[Proof of Theorem \ref{theorem hardy solutions are radial} given Theorem \ref{theorem moving planes}]

\emph{Step 1.} Let $z\neq 0$. We shall show that $u_z^*$ is strictly symmetric-decreasing with respect to any hyperplane passing through $0$ and $z$.

We want to deduce this from Theorem \ref{theorem moving planes} with $a=-z^*$, $v=u_z^*$ and $k=k_z$ applied to equation \eqref{equation for kelvin transform u*z}. Let us verify that the assumptions of this theorem are satisfied.

This is clear for the assumptions on $g$. Next, the function $k(x) = \frac{|z^*|}{|x| |x + z^*|}$ is strictly symmetric-decreasing with respect to any hyperplane passing through $0$ and $z$ (which is the same as passing through $0$ and $a=-z^*$). It decreases as $|x|^{-2}$ as $|x| \to \infty$ and behaves as $\text{dist}(x, S)^{-1}$ near the set $S$ given by \eqref{definition S}. Therefore the assumptions on $k$ are satisfied. Finally, according to Lemma \ref{lemma properties kelvin transform} the function $v=u_z^*$ belongs to $H_{\rm loc}^2(\R^n\setminus\{0,-z^*\})$ and is in $L^\frac{2n}{n-4}$ on any $\Omega$ which is a positive distance away from $0$ and $-z^*$. As in the proof of that lemma one sees that $u\in L^\frac{2n}{n-4}(\R^n)$ if and only if $u_z^*\in L^\frac{2n}{n-4}(\R^n)$.

Thus, the assertion of Step 1 follows from Theorem \ref{theorem moving planes}.

\emph{Step 2.} We now deduce that $u$ is radially symmetric-decreasing.

Let $l_z$ denote the line through $0$ and $z$. It follows from Step 1 that for any hyperplane $H$ orthogonal to $l_z$ the function $u^*_z$ is strictly radially symmetric-decreasing in $H$ with respect to the point $H\cap l_z$. By letting $z \to 0$ along a fixed direction $\nu\in\mathbb S^{n-1}$, we infer that in any hyperplane with normal $\nu$, $u^*$ is radially symmetric-decreasing with respect to the point where the hyperplane intersects the line through $0$ in direction $\nu$. Since $\nu$ is arbitrary, we conclude that $u^*$ is radially symmetric-decreasing with respect to $0$.

Applying what we have proved so far to $u^*$, which, by Lemma \ref{lemma properties kelvin transform}, is in $H^2_{\rm loc}(\R^n \setminus \{0\})$ if and only if $u$ is and which solves \eqref{differential eq} if and only if $u$ does, we also find that $u = (u^*)^*$ is radially symmetric-decreasing with respect to $0$.
\end{proof}


\subsection{Some integrability estimates}
\label{subsection integrability estimates}

In the previous subsection we have reduced the proof of Theorem \ref{theorem hardy solutions are radial} to the proof of Theorem \ref{theorem moving planes}.

In this and the next subsection we always assume that $g$ and $v$ satisfy the assumptions in Theorem \ref{theorem moving planes} and that \eqref{eq:symmeq} holds in the weak sense.

Our first step in proving Theorem \ref{theorem moving planes} is to understand the behavior of the solution near the singular points. This will later allow us to use a larger class of test functions, including functions whose support contains the set $S$ given by \eqref{definition S}. This amounts to proving integrability of $v$ at $S$ and constitutes the purpose of the present subsection. Our arguments in this subsection are inspired by \cite[Lemmas 3.1 and 3.2]{Yang}.

\begin{lemma}
\label{lemma v integrable}
We have $k^\frac{n+4}{2} g(k^{-\frac{n-4}{2}} v) \in L^1_{\rm loc}(\R^n)$ and $k^\frac{n+4}{2} g(k^{-\frac{n-4}{2}} v) \in L^\frac{2n}{n+4}(\Omega)$ for every $\Omega$ which is a positive distance away from $S$.
\end{lemma}

\begin{proof}
First, assume that $\Omega$ is a positive distance away from $S$. The inequality $g'(t)  \leq \frac{n+4}{n-4} \frac{g(t)}{t}$ in \eqref{conditions on g} implies by integration that $g(t)\leq g(1) t^{(n+4)/(n-4)}$ for $t\geq 1$ and therefore
$$
g(t) \leq g(1) t^{(n+4)/(n-4)} + \sup_{0\leq t\leq 1} g \lesssim t^{(n+4)/(n-4)} + 1 \qquad \text{ for all } t\geq 0 \,.
$$
Thus, $k^\frac{n+4}{2} g(k^{-\frac{n-4}{2}} v) \lesssim v^\frac{n+4}{n-4} + k^\frac{n+4}2$ and consequently
\begin{align*}
\int_\Omega \left( k^\frac{n+4}{2} g(k^{-\frac{n-4}{2}} v) \right)^\frac{2n}{n+4} & \lesssim \int_{\Omega} \left( v^\frac{2n}{n-4} + k^n \right) \,.
\end{align*}
The integral on the right side is finite by our assumptions.

In order to prove the local integrability, let $\eta$ be a smooth, non-negative function on $\R^n$ with $\eta\equiv 1$ near $0$ and with support in a ball not containing $a$. We shall show that $\int k^\frac{n+4}{2} g(k^{-\frac{n-4}{2}} v)\eta<\infty$. This and a corresponding assertion for $\eta$ supported near $a$ proves the lemma.

For $\epsilon>0$ so small that $\eta\equiv 1$ on $\{|x|\leq 2\epsilon\}$, we can find non-negative cut-off functions $\eta_\epsilon \in C^\infty_0$ such that 
\[ \eta_\epsilon (x)
\begin{cases}
\equiv 0 & \text{ if } |x| \leq \epsilon, \\
\equiv \eta & \text{ if } |x| \geq 2 \epsilon \,,
\end{cases}
\]
and
\[ |D^k \eta_\epsilon(x) | \lesssim \epsilon^{-k} 
\text{ for all } k= 1,2,3,4 \,.
\]

For $q$ as in assumption \eqref{conditions on g}, let $m = \frac{4q}{q-1}$ and define $\xi_\epsilon = (\eta_\epsilon)^m$. Since $\xi_\epsilon$ is supported away from $S$, it is a valid test function for equation \eqref{eq:symmeq}, and we obtain
\begin{align*}
\int_{\R^n} k^\frac{n+4}{2} g(k^{-\frac{n-4}{2}} v ) \xi_\epsilon = \int_{\R^n} \Delta v \Delta \xi_\epsilon = \int_{\R^n} v \Delta^2 \xi_\epsilon. 
\end{align*}
Observing that 
\[ |\Delta^2 \xi_\epsilon| \lesssim \epsilon^{-4} \eta_\epsilon^{m-4} 1_{\{\epsilon < |x| < 2 \epsilon\}} + \Delta^2 \eta^m = \epsilon^{-4} \xi_\epsilon^{1/q} 1_{\{\epsilon < |x| < 2 \epsilon\}} + \Delta^2 \eta^m \,, \]
and using the assumed lower bound on $k$ (note that $q\leq\frac{n+4}{n-4}$, since $g(t) \lesssim t^{(n+4)/(n-4)}$), we find that 
\begin{align*}  \int_{\R^n} k^\frac{n+4}{2} g(k^{-\frac{n-4}{2}} v) \xi_\epsilon & \lesssim \epsilon^{-4} \int_{\{\epsilon < |x| < 2 \epsilon\}} v \xi_\epsilon^{1/q} + \int v \Delta^2 \eta^m    \\
& \lesssim  \epsilon^{-4 + \frac{n+4}{2q} - \frac{n-4}{2}} \int_{\{\epsilon < |x| < 2 \epsilon\}} v k^{-\frac{n-4}{2} + \frac{n+4}{2q}} \xi_\epsilon^{1/q} + \int v \Delta^2 \eta^m  \\
& \lesssim  \epsilon^{-4 + \frac{n+4}{2q} - \frac{n-4}{2} + n\frac{q-1}{q}} \big( \int_{\{\epsilon < |x| < 2 \epsilon\}} v^q k^{-q \frac{n-4}{2} + \frac{n+4}{2}} \xi_\epsilon\big)^\frac{1}{q} + \int v \Delta^2 \eta^m \,.
\end{align*}
Because of the equivalence
\begin{equation}
\label{exponent of epsilon equivalence}
 -4 + \frac{n+4}{2q} - \frac{n-4}{2} + n\frac{q-1}{q} > 0 \quad \Leftrightarrow \quad q > 1 \,,
\end{equation} 
the exponent of $\epsilon$ is positive and therefore we can drop the factor in front of the first term. The second term is finite since $\Delta^2\eta^m=0$ near zero, and independent of $\epsilon$. To close the estimate, we will use the fact that
$$
t^q \leq c^{-1} g(t) +1 \lesssim g(t) +1
\text{ for all } t\geq 0 \,.
$$
Using this, we can estimate
\begin{align*} & \big( \int_{\{\epsilon < |x| < 2 \epsilon\}} v^q k^{-q \frac{n-4}{2} + \frac{n+4}{2}} \xi_\epsilon \big)^\frac{1}{q}  \\
\lesssim \big( & \int_{\{\epsilon < |x| < 2 \epsilon\} }  k^\frac{n+4}{2} g(k^{-\frac{n-4}{2}} v) \xi_\epsilon +  \int_{\{\epsilon < |x| < 2 \epsilon\} } k^\frac{n+4}{2} \big) ^\frac{1}{q}. 
\end{align*}
The second term on the right side is finite by assumption. We have thus proved that 
\[ \int_{\R^n} k^\frac{n+4}{2} g(k^{-\frac{n-4}{2}} v) \xi_\epsilon \lesssim \big(  \int_{\R^n} k^\frac{n+4}{2} g(k^{-\frac{n-4}{2}} v) \xi_\epsilon \big)^\frac{1}{q} + 1 \,, \]
which implies that 
\[ \int_{\R^n} k^\frac{n+4}{2} g(k^{-\frac{n-4}{2}} v) \xi_\epsilon \lesssim 1. \]
Letting $\epsilon \to 0$, we conclude by monotone convergence that 
\[  \int_{\R^n} k^\frac{n+4}{2} g(k^{-\frac{n-4}{2}} v)\eta  < \infty, \]
which finishes the proof.
\end{proof}

We can use the fundamental integrability properties from Lemma \ref{lemma v integrable} to enlarge the class of functions one can test equation \eqref{eq:symmeq} against. This is the content of the next lemma. 

We recall that by definition $\dot H^2(\R^n)$ is the completion of $C_0^\infty(\R^n)$ with respect to $\|\Delta u\|_2$, see e.g. \cite{Mazya}.

\begin{lemma}
\label{lemma v weak solution}
Let $\varphi \in \dot H^2(\R^n)$ and assume, in addition, that $\Delta\varphi\equiv 0$ in a neighborhood of $S$. Then 
\[ \int_{\R^n} \Delta v \Delta \varphi = \int_{\R^n}  k^\frac{n+4}{2} g(k^{-\frac{n-4}{2}} v) \varphi. \]
\end{lemma}

\begin{proof}
We begin by localizing the problem. We choose non-negative $C^\infty$ functions $\chi_0$, $\chi_a$ and $\chi_\infty$ such that $\chi_0+\chi_a+\chi_\infty\equiv 1$ on $\R^n$ and such that $\chi_0$ and $\chi_a$ are $\equiv 1$ near the points $0$ and $a$, respectively, and both have compact support.

Given $\varphi$ as in the lemma, it clearly suffices to prove the theorem for each of the functions $\chi_0\varphi$, $\chi_a\varphi$ and $\chi_\infty\varphi$. Note that all three functions belong to $\dot H^2(\R^n)$ and are harmonic near $S$.

The identity for $\chi_\infty\varphi$ follows by a straightforward approximation argument using the fact that on the support of $\chi_\infty\varphi$, $\Delta v$ is in $L^2$ and, by Lemma \ref{lemma v integrable}, $k^{\frac{n+4}{2}} g(k^{-\frac{n-4}{2}}v)$ is in $L^\frac{2n}{n+4}$.

The argument for $\chi_a\varphi$ is the same as that for $\chi_0\varphi$, so we focus on the latter. To ease notation, we write $\varphi$ instead of $\chi_0\varphi$ and assume that $\varphi$ has compact support not containing $a$ and is harmonic near $0$. By harmonicity, $\varphi$ is $C^\infty$ near zero and, in particular, $\varphi$ and all its derivatives are bounded in a ball $B$ near zero.

For $\epsilon >0$ so small that $B(0, 2 \epsilon) \subset B$, fix $\eta_\epsilon \in C^\infty(\R^n)$ such that
\begin{equation}
\label{definition eta cutoff function} 
\eta_\epsilon (x)
\begin{cases}
\equiv 0 & \text{ if } |x| \leq \epsilon, \\
\equiv 1 & \text{ if } |x| \geq 2 \epsilon, \\
\end{cases}
\end{equation}
and $|D^k \eta_\epsilon| \lesssim \epsilon^{-k}$ for $k = 1,2,3,4$. We can now test the equation for $v$ with $\varphi \eta_\epsilon$, which is a valid test function, since it is in $\dot H^2(\R^n)$ and supported away from $S$. We obtain
\[ \int_{\R^n} \Delta v \Delta (\varphi \eta_\epsilon) = \int_{\R^n} k^\frac{n+4}{2} g(k^{-\frac{n-4}{2}} v) \varphi \eta_\epsilon. \]
As $\epsilon\to 0$, the right hand side tends to $\int_{\R^n} k^\frac{n+4}{2} g(k^{-\frac{n-4}{2}} v) \varphi$ by dominated convergence, using $\varphi \in L^\infty(B)$ and Lemma \ref{lemma v integrable}. To evaluate the left hand side, we write 
\begin{equation}
\label{three terms delta} \Delta (\varphi \eta_\epsilon) = \eta_\epsilon \Delta \varphi + 2 \nabla \varphi \cdot \nabla \eta_\epsilon + \varphi \Delta \eta_\epsilon \,.
\end{equation}
For the first term on the right side, since $\Delta \varphi = 0$ in $\{\eta_\epsilon\neq 1\}$, we have
\[ \int_{\R^n} (\Delta v \Delta \varphi) \eta_\epsilon = \int_{\R^n} \Delta v \Delta \varphi \,.\]
Therefore, to finish the proof, it remains to show that 
\begin{equation}
\label{remainder terms tend to zero} \lim_{\epsilon \to 0} \int_{\R^n} \Delta v (2 \nabla \varphi \cdot \nabla \eta_\epsilon + \varphi \Delta \eta_\epsilon) = 0. 
\end{equation}
Using the facts that $\varphi \in \dot H^2(\R^n)$ and that $\Delta \varphi \equiv 0$ on $\text{supp}(\nabla \eta_\epsilon) \cup \text{supp}(\Delta \eta_\epsilon)  \subset \{ \epsilon < |x| < 2 \epsilon\} \subset B$, we obtain from integration by parts that 
\[ \int_{\R^n} \Delta v (2 \nabla \varphi \cdot \nabla \eta_\epsilon + \varphi \Delta \eta_\epsilon) = \int_{\{ \epsilon < |x| < 2 \epsilon\}} v \big( (4 \sum_{i,j =1}^n \partial_{ij} \varphi \partial_{ij} \eta_\epsilon) + 4  \nabla \varphi \cdot \nabla \Delta \eta_\epsilon + \varphi \Delta^2 \eta_\epsilon \big). \]
Similarly to the proof of Lemma \ref{lemma v integrable}, since $k \gtrsim \epsilon^{-1}$ on $\{ \epsilon < |x| < 2 \epsilon \}$ and since $\partial_{ij} \varphi$ is bounded on  $\{ \epsilon < |x| < 2 \epsilon \}$, we have
\begin{align*} 
 &\Big | \int_{\{ \epsilon < |x| < 2 \epsilon\}} v \sum_{i,j =1}^n \partial_{ij} \varphi \partial_{ij} \eta_\epsilon \Big | \\
  & \lesssim \epsilon^{-\frac{n-4}{2} + \frac{n+4}{2q}} \big( \int_{\{ \epsilon < |x| < 2 \epsilon\}} k^\frac{n+4}{2}(1 + g(k^{-\frac{n-4}{2}} v)) \big)^\frac{1}{q}  \big( \int_{\{ \epsilon < |x| < 2 \epsilon\}} \sum_{i,j=1}^n |\partial_{ij} \eta_\epsilon|^\frac{q}{q-1} \big)^\frac{q-1}{q} 
\end{align*}
by Hölder's inequality. From Lemma \ref{lemma v integrable} and the bound $|D^2 \eta_\epsilon| \lesssim \epsilon^{-2}$, we infer that
\begin{align*} 
 \Big | \int_{\{ \epsilon < |x| < 2 \epsilon\}}  v \sum_{i,j =1}^n \partial_{ij} \varphi \partial_{ij} \eta_\epsilon \Big|  \lesssim \epsilon^{-2 + \frac{n(q-1)}{q} -\frac{n-4}{2} + \frac{n+4}{2q}}.
\end{align*}
By \eqref{exponent of epsilon equivalence}, we have $-2 + \frac{n(q-1)}{q} -\frac{n-4}{2} + \frac{n+4}{2q} > 0$ and we conclude that 
\[ \lim_{\epsilon \to 0} \int_{\R^n}  v  \sum_{i,j =1}^n \partial_{ij} \varphi \partial_{ij} \eta_\epsilon  = 0. \]

By an analogous argument, using boundedness of $\nabla \varphi$ and $\varphi$ on $B$ and the bounds $|D^k \eta_\epsilon| \leq \epsilon^{-k}$ for $k = 3,4$, one can establish that
\[  \lim_{\epsilon \to 0} \int_{\R^n}  v (\nabla \varphi \cdot \nabla \Delta \eta_\epsilon) = \lim_{\epsilon \to 0} \int_{\R^n}  v  \varphi \Delta^2 \eta_\epsilon  = 0 \,. \]
The proof of \eqref{remainder terms tend to zero}, and therefore of Lemma \ref{lemma v weak solution}, is complete. 
\end{proof}


\subsection{Proof of Theorem \ref{theorem moving planes}}
\label{subsection moving planes}

In this subsection we prove Theorem~\ref{theorem moving planes} using the method of moving planes in a variant relying mostly on integral estimates, the crucial ones being derived in Lemma \ref{lemma 4.4} below. The use of such bounds in the context of the method of moving planes goes back at least to \cite{Te}. In the present context of a fourth-order equation, this strategy is however much harder to implement because more regularity is required from admissible test functions, compare Lemma \ref{lemma v weak solution}. We achieve this by a careful regularization procedure together with odd reflection across the hyperplane $\{ x_1 = \lambda \}$. This is carried out in the proofs of Lemmas \ref{lemma inequality for w} and \ref{lemma 4.4} below.

\textbf{Conventions and notations.}  Again we assume that $g$ and $v$ satisfy the assumptions of Theorem \ref{theorem moving planes} and that \eqref{eq:symmeq} holds in the weak sense. Recall moreover that we denote $S = \{0 ,a\} \subset \R^n$ for some fixed $a \in \R^n \setminus \{0\}$. 

Since the assumptions and the conclusion of the theorem are invariant under rotations, we may assume that $a_1=0$ and $H=\{x_1=0\}$.

For any number $\lambda < 0$, we introduce the moving planes notation by letting $\Sigma_\lambda = \{ x_1 > \lambda \}$, $x^{\lambda} = (2 \lambda - x_1, x_2, ..., x_n)$, $v_\lambda(x) = v(x^\lambda)$ and $k_\lambda(x) = k(x^\lambda)$. (This should not be confused with the function $k_z$ from Subsection \ref{sec:kelvin}.) Moreover, on $\SL$, we define the difference function $w^{(\lambda)} = v - v_\lambda$. We will consider this function only in the half-space $\Sigma_\lambda$. When $\lambda$ is understood, we will often abbreviate this function by $w$.

For any function $u$, we denote by $u_- := \max \{ 0, -u\}$ its negative part (note that with our convention $u_- \geq 0$). 

The following lemma lies at the core of the moving planes argument used to prove Theorem \ref{theorem moving planes}. 

We recall that by definition $\dot W^{1,\frac{2n}{n-2}}_0(\SL)$ is the completion of $C_0^\infty(\SL)$ with respect to $\|\nabla u\|_{\frac{2n}{n-2}}$, see e.g. \cite{Mazya}.

\begin{lemma}
\label{lemma inequality for w}
Let $v$, $g$ and $k$ fulfill the assumptions of Theorem \ref{theorem moving planes} and let $w^{(\lambda)} = v - v_\lambda$ on $\Sigma_\lambda$. Define, for $x \in \SL$,
\begin{equation}
\label{definition V}
V^{(\lambda)}(x) = k_\lambda(x)^4 \; \frac{g(v(x) k_\lambda(x)^{-\frac{n-4}{2}}) - g(v_\lambda(x) k_\lambda(x)^{-\frac{n-4}{2}})}{v(x) k_\lambda(x)^{-\frac{n-4}{2}} - v_\lambda(x) k_\lambda(x)^{-\frac{n-4}{2}}}. 
 \end{equation}
 Then the following holds. 
\begin{enumerate}
\item[(a)] We have $V^{(\lambda)} \geq 0$ and $V^{(\lambda)} \in L^\frac{n}{4}(\{ w^{(\lambda)} < 0\})$ for all $\lambda < 0$. Moreover,
$$
\lim_{\lambda \to -\infty} \int_{\{w^{(\lambda)} < 0\}} \left( V^{(\lambda)} \right)^\frac{n}{4} = 0 \,.
$$ 
\item[(b)] 
Let $0\leq \psi \in \dot H^2(\SL) \cap \dot W^{1,\frac{2n}{n-2}}_0(\SL)$ with $\Delta \psi \equiv 0$ in a neighborhood of $S$. Then
\begin{equation}
\label{inequality for w} 
\int_{\Sigma_\lambda} \Delta w^{(\lambda)} \Delta \psi \geq - \int_{\Sigma_\lambda} V^{(\lambda)} (w^{(\lambda)})_- \psi .
\end{equation} 
\end{enumerate}
\end{lemma}

\begin{proof}
As a rule, we will abbreviate $w=w^{(\lambda)}$ and $V=V^{(\lambda)}$.

\textit{Proof of (a).   }
Firstly, since $k >0$ and $g' \geq 0$ by \eqref{conditions on g}, we have $V \geq 0$. 

Secondly, as we have seen in the proof of Lemma \ref{lemma v integrable}, $g(t)
\lesssim t^\frac{n+4}{n-4}$ for $t\geq 1$. Reinserting this bound into the assumed upper bound on $g'$ and using the fact that $g'$ is bounded on $(0,1]$ (since $\lim_{t \to 0} g'(t)$ exists and is finite), we obtain the bound 
\[ g'(t) \lesssim 1 + t^\frac{8}{n-4}. \]
Therefore, by the mean value theorem,
\begin{equation}
\label{difference quotient bound on g} 
\frac{g(t)  - g(s)}{t -s} \lesssim 1 + t^\frac{8}{n-4}  \quad \text{ for all } 0 < s < t. 
\end{equation}
Applying \eqref{difference quotient bound on g} with $t = v_\lambda(x) k_\lambda(x)^{-\frac{n-4}{2}}$ and $s = v(x) k_\lambda(x)^{-\frac{n-4}{2}}$ and noticing that $t > s$ whenever $w<0$, we can bound
\begin{align*}
&\int_{\{ w < 0\}} V^\frac{n}{4} \lesssim \int_{\{ w < 0\}} k_\lambda^n + \int_{\{ w < 0\}} v_\lambda^\frac{2n}{n-4} 
\leq  \int_{\{ x_1 < \lambda\}} \left( k^n + v^\frac{2n}{n-4}\right).
\end{align*}
The right side is finite by the integrability assumptions on $k$ and $v$ and, by dominated convergence, tends to zero as $\lambda\to-\infty$.

\textit{Proof of (b).    }
Let $\psi \in \dot H^2(\SL) \cap \dot W^{1,\frac{2n}{n-2}}_0(\SL)$ fulfill the assumptions of Lemma \ref{lemma inequality for w}. Then the odd extension of $\psi$,
\[ \varphi(x) = 
\begin{cases}
\psi(x) & \text{ if } x \in \SL \,, \\
-\psi (x^\lambda) & \text{ if } x \in \SL^c \,,
\end{cases}
\]
belongs to $\dot H^2(\R^n)$ and is harmonic near $S$. Therefore, Lemma \ref{lemma v weak solution} and a straightforward change of variables yield
\begin{align*}
& \int_{\SL} \Delta w \Delta \psi = \int_{\R^n} \Delta v\Delta\varphi = \int_{\R^n} k^{\frac{n+4}{2}} g(k^{-\frac{n-4}{2}}v)\varphi \\
=& \int_{\SL} \Big( k^\frac{n+4}{2} g(k^{-\frac{n-4}{2}} v) - k_\lambda^\frac{n+4}{2} g(v_\lambda k_\lambda^{-\frac{n-4}{2}}) \Big ) \psi \\
= & \int_{\SL} V w \psi + \int_{\SL} \Big( k^\frac{n+4}{2} g(k^{-\frac{n-4}{2}} v) - k_\lambda^\frac{n+4}{2} g(v k_\lambda^{-\frac{n-4}{2}}) \Big ) \psi \\
 \geq & \int_{\SL} V w \psi \geq - \int_{\SL} V w_- \psi.
\end{align*}
In the first inequality, we used $k \geq k_\lambda$ on $\SL$ together with the fact that the function $s \mapsto s^\frac{n+4}{2} g(v s^{-\frac{n-4}{2}})$ is non-decreasing on $(0, \infty)$ for every fixed $v > 0$, which follows from the inequality $g'(t) \leq \frac{n+4}{n-4} \frac{g(t)}{t}$ in assumption \eqref{conditions on g}. In the second inequality, we used $V \geq 0$ and $\psi \geq 0$. 
\end{proof}

We can now derive the crucial technical ingredient for the moving planes method from inequality \eqref{inequality for w}. 

\begin{lemma}
\label{lemma 4.4}
Let $v$, $g$ and $k$ fulfill the assumptions of Theorem \ref{theorem moving planes} and let $w^{(\lambda)} = v - v_\lambda$ on $\SL$. Let $V^{(\lambda)}$ be defined by \eqref{definition V}. Then there is $\epsilon_0 > 0$, depending only on $n$, such that if $|\{ w^{(\lambda)}<0\}|>0$, then $\;\int_{\{w^{(\lambda)} < 0\}} \left( V^{(\lambda)} \right)^\frac{n}{4} \geq \epsilon_0$.
\end{lemma}

\begin{proof}
We abbreviate $w=w^{(\lambda)}$ and $V=V^{(\lambda)}$. We claim that the assertion follows if we can prove the following two inequalities,
\begin{align}
\big( \int_{\SL} w_-^\frac{2n}{n-4}\big) ^\frac{n-4}{2n} & \lesssim \big( \int_{\SL} (-\Delta w)_-^2 \big) ^\frac{1}{2}, \label{inequality w leq Delta w} \\ 
\big( \int_{\SL} (-\Delta w)_-^2 \big) ^\frac{1}{2} & \lesssim \big( \int_{\{w < 0\}} V^\frac{n}{4} \big) ^\frac{4}{n} \big( \int_{\SL} w_-^\frac{2n}{n-4}\big) ^\frac{n-4}{2n}, \label{inequality Delta w leq w}
\end{align}
with implied constants depending only on $n$ and where
\[ (-\Delta w)_-^2 := ((-\Delta w)_-)^2. \]
Indeed, \eqref{inequality w leq Delta w} and \eqref{inequality Delta w leq w} together yield the bound 
  \[  \big( \int_{\SL} w_-^\frac{2n}{n-4}\big) ^\frac{n-4}{2n} \lesssim \big( \int_{\{w < 0\}} V^\frac{n}{4} \big) ^\frac{4}{n} \big( \int_{\SL} w_-^\frac{2n}{n-4}\big) ^\frac{n-4}{2n}. \]
If $w_- \nequiv 0$, we may divide by $ \big( \int_{\SL} w_-^\frac{2n}{n-4}\big) ^\frac{n-4}{2n} \neq 0$ to deduce the bound 
\[ 1 \lesssim \big( \int_{\{w < 0\}} V^\frac{n}{4} \big) ^\frac{4}{n}, \]
which concludes the proof of Lemma \ref{lemma 4.4}. Thus, it remains to prove inequalities \eqref{inequality w leq Delta w} and \eqref{inequality Delta w leq w}. 

Before proving these inequalities, let us note that the second factor on the right side of \eqref{inequality Delta w leq w} is finite since $w\geq -v_\lambda$ and $v_\lambda\in L^\frac{2n}{n-4}(\Sigma_\lambda)$ by our assumption on $v$. Moreover, by Lemma \ref{lemma inequality for w} the first factor on the right side of \eqref{inequality Delta w leq w} is finite and therefore it is part of the assertion of this inequality that $(-\Delta w)_- \in L^2(\Sigma_\lambda)$.

For any $f \in L^2(\Sigma_\lambda)$, we can define 
\begin{equation}
\label{psi convolution} 
\psi(x) := c_n  \int_{\SL} \big( \frac{1}{|x - y|^{n-2}} - \frac{1}{|x-y^\lambda|^{n-2}} \big) f (y), \qquad x \in \SL,
\end{equation}
with $c_n = \big((n-2)|\mathbb S^{n-1}|\big)^{-1}$. Then we have $\psi \in \dot H^2(\SL) \cap \dot W^{1,\frac{2n}{n-2}}_0(\SL)$ and $-\Delta \psi = f$ on $\SL$.  

Notice that $f \leq 0$ implies that $\psi \leq 0$ on $\SL$. Moreover, formula \eqref{psi convolution} and the Hardy--Littlewood--Sobolev inequality \cite[Theorem 4.3]{LiLo} imply the bound
\begin{equation}
\label{hls bound for psi} 
\| \psi \|_{L^r(\SL)} \lesssim \| f \|_{L^s(\SL)} 
\end{equation} 
for every pair of exponents $1< s < r <\infty$ related by $\frac{n-2}{n} + \frac{1}{s} = 1 + \frac{1}{r}$.  

We now give the proofs of the two inequalities \eqref{inequality w leq Delta w} and \eqref{inequality Delta w leq w}.

To prove \eqref{inequality w leq Delta w}, for every $k \in \N$ and $x \in \SL \setminus S$, set 
\[ f_k(x) = - w_-^\frac{n+4}{n-4}(x) 1_{\{ w_- \leq k \}}(x) 1_{\{\text{dist}(x, S) \geq k^{-1}, |x| \leq k \}}(x).
\]
Thus, $f_k$ is bounded and has compact support and, in particular, $f_k \in L^2(\SL)$. We consider $\psi_k$ associated to $f_k$ by \eqref{psi convolution}. Notice that $-\Delta \psi_k = f_k \equiv 0$ in a neighborhood of $S$. 

We have, by dominated (or monotone) convergence,
\[ \int_{\SL} w_-^\frac{2n}{n-4} =  \lim_{k \to \infty} \int_{\SL} w f_k = \lim_{k \to \infty} \int_{\SL} w (-\Delta \psi_k). \] 
Let us introduce a family of non-negative cutoff functions $\eta_\epsilon$ as in \eqref{definition eta cutoff function}, but now vanishing both near $0$ and $a$. Since $\psi_k$ and $w$ vanish on $\partial \SL$, integration by parts gives
\begin{equation}
\label{integration by parts with cutoff, lemma 4.4} \int_{\SL} w (-\Delta (\psi_k \eta_\epsilon)) = \int_{\SL} (-\Delta w) (\psi_k \eta_\epsilon)  \leq -\int_{\SL} (-\Delta w)_- (\psi_k \eta_\epsilon).
\end{equation}
Here the inequality holds simply because $\psi_k \leq 0$. As $\epsilon \to 0$, the right hand side of \eqref{integration by parts with cutoff, lemma 4.4} tends to $-\int_{\SL} (-\Delta w)_- \psi_k$, by monotone convergence. Moreover, since $\psi_k$ is harmonic in a neighborhood of $S$, we can argue as in the proof of Lemma \ref{lemma v weak solution} that the left hand side of \eqref{integration by parts with cutoff, lemma 4.4} tends to $\int_{\SL} w (-\Delta \psi_k)$ as $\epsilon \to 0$. We therefore have that 
\[ \int_{\SL} w (-\Delta \psi_k) \leq - \int_{\SL} (-\Delta w)_- \psi_k \leq \|(-\Delta w)_-\|_2 \| \psi_k\|_2 \lesssim  \|(-\Delta w)_-\|_2  \| f_k\|_\frac{2n}{n+4}. \]
The second inequality here is Hölder's inequality and the third one uses \eqref{hls bound for psi} with $r = 2$, $s = \frac{2n}{n+4}$. Letting $k \to \infty$, we find that 
\[ \int_{\SL} w_-^\frac{2n}{n-4} \lesssim  \|(-\Delta w)_-\|_2  \| w_-^\frac{n+4}{n-4}\|_\frac{2n}{n+4} =  \|(-\Delta w)_-\|_2  \| w_-\|_\frac{2n}{n-4}^\frac{n+4}{n-4}, \]
which implies inequality \eqref{inequality w leq Delta w} since $\| w_-\|_\frac{2n}{n-4}<\infty$ as we have noted before.

To prove \eqref{inequality Delta w leq w}, for every $k \in \N$ and $x \in \SL \setminus S$, set 
\[ f_k(x)= - (-\Delta w)_-(x) 1_{\{ (-\Delta w)_- \leq k \}} 1_{\{\text{dist}(x, S) \geq k^{-1}, |x| \leq k\}}.\]
Thus, $f_k$ is bounded and has compact support and, in particular, $f_k \in L^2(\SL)$. We consider the associated $\psi_k$ as above and notice that $-\Delta \psi_k = f_k \equiv 0$ in a neighborhood of $S$. 

We have, by monotone convergence,
\[ \int_{\SL} (-\Delta w)_-^2 = - \lim_{k \to \infty} \int_{\SL} \Delta w  f_k = \lim_{k \to \infty} \int_{\SL} \Delta w \Delta \psi_k. \] 

Since $- \psi_k \geq 0$ fulfills all the assumptions of Lemma \ref{lemma inequality for w}, by \eqref{inequality for w} we obtain
\begin{align*}
 \int_{\SL} \Delta w \Delta \psi_k \leq - \int_{\SL} V w_- \psi_k &\leq \|V 1_{\{w < 0\}} \|_\frac{n}{4} \|w_- \|_\frac{2n}{n-4} \|\psi_k\|_\frac{2n}{n-4} \\
 & \lesssim \|V 1_{\{w < 0\}} \|_\frac{n}{4} \|w_- \|_\frac{2n}{n-4} \| f_k\|_2, 
 \end{align*}
where we used Hölder's inequality followed by \eqref{hls bound for psi} again. Since 
\[ \int_{\SL} \Delta w \Delta \psi_k = \| f_k\|_2^2 <\infty, \]
we deduce that 
\[ \| f_k \|_2 \lesssim  \|V 1_{\{w < 0\}} \|_\frac{n}{4}  \|w_- \|_\frac{2n}{n-4} \]
for all $k \in \N$. Passing to the limit $k \to \infty$, we obtain inequality \eqref{inequality Delta w leq w} by monotone convergence. 
\end{proof}

We are finally in a position to prove Theorem \ref{theorem moving planes}. 
 
 \begin{proof}
[Proof of Theorem \ref{theorem moving planes}]
As we have already mentioned, by rotation invariance we may assume that $a_1=0$ and $H=\{x_1=0\}$.
For $\lambda < 0$, we consider the function $w^{(\lambda)} = v - v_\lambda$, defined on $\SL$, and $V^{(\lambda)}$ given by \eqref{definition V}. 
By Lemma \ref{lemma inequality for w}, we have $\int_{\{w^{(\lambda)} < 0\}} (V^{(\lambda)})^\frac{n}{4} < \epsilon_0$ for all $\lambda$ sufficiently negative. Therefore Lemma \ref{lemma 4.4} implies that $w^{(\lambda)} \geq 0$ on $\SL$ for all $\lambda$ sufficiently negative.  

Therefore
\[ \lambda_0 = \sup \{ \lambda < 0 \, : \, w^{(\mu)} \geq 0 \text{  for all  } \mu < \lambda \}, \]
is well-defined.

We claim that either $w_{\lambda_0}>0$ a.e. on $\Sigma_{\lambda_0}$ or $w_{\lambda_0}\equiv 0$. Indeed, by continuity, we still have $w^{(\lambda_0)} \geq 0$ a.e. in $\Sigma_{\lambda_0}$. Moreover, inequality \eqref{inequality Delta w leq w} implies that $-\Delta w^{(\lambda)}\geq 0$ in $\SL$ for $\lambda<\lambda_0$ and, therefore, by continuity $-\Delta w^{(\lambda_0)} \geq 0$ in $\Sigma_{\lambda_0}$. (For the continuity argument, we pass to the limit in the inequality $\int_{\SL} w^{(\lambda)}(-\Delta)\varphi\geq 0$ for $0\leq\varphi\in C_0^\infty(\SL)$.) The claim now follows by the strong maximum principle in $\Sigma_{\lambda_0}$.

After these preliminaries we now show that if $\lambda_0 < 0$, then $w^{(\lambda_0)}\equiv 0$. Later we will see that $w^{(\lambda_0)}\equiv 0$ is impossible and therefore we will conclude that $\lambda_0 = 0$. We argue by contradiction and assume that $\lambda_0 < 0$ and that $w^{(\lambda_0)}\nequiv 0$. Then by the above argument $w^{(\lambda_0)} > 0$ a.e. on $\Sigma_{\lambda_0}$. This strict inequality implies that the quantity
\[ I(\lambda) := \int_{\{ x \in \SL \; : \; w^{(\lambda)}(x) < 0 \}} \left( v_\lambda^\frac{2n}{n-4} + k_\lambda ^n \right) \]
tends to zero as $\lambda \searrow \lambda_0$. Indeed, any sequence of $\lambda$'s tending to $\lambda_0$ has a subsequence along which $w^{(\lambda)}(x) \to w^{(\lambda_0)}(x)$ pointwise for a.e. $x \in \Sigma_{\lambda_0}$. Since $w^{(\lambda_0)} > 0$, we have that $1_{\{w^{(\lambda)} < 0\}}(x) \to 0$ pointwise for a.e. $x \in \Sigma_{\lambda_0}$, which is the same as $1_{\{v(x) > v(x^\lambda)\}}\to 0$ pointwise almost everywhere in $\{x_1<\lambda_0\}$. Therefore, by dominated convergence, in view of the integrability assumptions on $v$ and $k$ and the assumption $\lambda_0<0$,
\begin{align*}
I(\lambda) = \int_{\{x_1 < \lambda\}} 1_{\{v(x) > v(x^\lambda)\}}(x) ( v^\frac{2n}{n-4}(x) + k^n(x) ) \to 0 \,.
\end{align*} 

On the other hand, as in the proof of Lemma \ref{lemma inequality for w}, we have $ \int_{\{w^{(\lambda)} < 0 \}} (V^{(\lambda)})^\frac{n}{4} \lesssim I(\lambda)$ for all $\lambda<0$. This, together with $I(\lambda)\to 0$ as $\lambda \searrow \lambda_0$, implies by Lemma \ref{lemma 4.4} that there is a $\delta > 0$ such that $w^{(\lambda)} \geq 0$ for all $\lambda \in (\lambda_0, \lambda_0 + \delta)$. This is a contradiction to the definition of $\lambda_0$, and therefore we conclude that $w^{(\lambda_0)} \equiv 0$ if $\lambda_0<0$.

We now show that $w^{(\lambda_0)}\equiv 0$ for $\lambda_0<0$ is impossible under the assumptions of the theorem. Indeed, if $v\equiv v_{\lambda_0}$ and $\lambda_0<0$, then, by the integrability assumption on $v$,
$$
\int_{\R^n} v^\frac{2n}{n-4} = \int_{x_1<\lambda_0} v^\frac{2n}{n-4} +
\int_{x_1>\lambda_0} v^\frac{2n}{n-4} = 2 \int_{x_1<\lambda_0} v^\frac{2n}{n-4} <\infty \,.
$$
This is impossible under the assumption $v\not\in L^\frac{2n}{n-4}(\R^n)$. Also, if $v\equiv v_{\lambda_0}$, then by \eqref{eq:symmeq}, $k^\frac{n+4}{2} g(k^{-\frac{n-4}{2}}v)\equiv k_\lambda^\frac{n+4}{2} g(k_\lambda^{-\frac{n-4}{2}}v)$. Thus,
$$
0 = k^\frac{n+4}{2} g(k^{-\frac{n-4}{2}}v)- k_\lambda^\frac{n+4}{2} g(k_\lambda^{-\frac{n-4}{2}}v) = \frac{n-4}{2} \int_{k_\lambda}^k \left( \frac{n+4}{n-4} \frac{g(s^{-\frac{n-4}{2}} v)}{s^{-\frac{n-4}{2}} v} - g'(s^{-\frac{n-4}{2}} v) \right)\! v s^3
$$
and when the integrand is positive for almost every $s$, then we deduce that $k\equiv k_\lambda$. To summarize, under the assumptions of the theorem it is impossible that $w^{(\lambda_0)}\equiv 0$ for $\lambda_0<0$. Thus, we conclude that $\lambda_0=0$.

The fact that $\lambda_0 = 0$ implies that 
\[ v(x_1,x_2,...,x_n) \geq v(-x_1, x_2,...,x_n) \text{  for all  } x_1 > 0, x_2,...,x_n \in \R \]
and that $v$ is strictly increasing in $x_1$ on $\{x_1 < 0 \}$. (The strictness follows from the fact that $w^{(\lambda)}>0$ for $\lambda<0$.)

Considering $\tilde{v}(x_1,x_2,...,x_n) := v(-x_1, x_2,...,x_n)$, which solves the same equation as $v$ due to the symmetry of $k$, and repeating all of the previous arguments, we derive the complementary inequality
\[ v(x_1,x_2,...,x_n) \leq v(-x_1, x_2,...,x_n) \text{  for all  } x_1 > 0, x_2,...,x_n \in \R \]
and the fact that $v$ is strictly decreasing in $x_1$ on $\{x_1 >0 \}$.

Hence $v$ is strictly symmetric-decreasing with respect to $\{x_1=0\}$, as claimed.
\end{proof}


\section{ODE analysis}\label{section ode}

In this section, we prove Theorem \ref{theorem classification ODE} as well as Remark \ref{remark strict}. 

Recall from the introduction that by the radial symmetry of any solution $u$ on $\R^n \setminus \{0\}$ of \eqref{differential eq} proved in Theorem \ref{theorem hardy solutions are radial}, we can define a function $v$ on $\R$ by setting $u(x) = |x|^{-\frac{n-4}{2}} v(\ln |x|)$. Via this Emden-Fowler change of variables, equation \eqref{differential eq} is equivalent to 
\begin{equation}
\label{ode with g}
v^{(4)} - A v'' + B v = g(v)  \qquad \text{in } \R
\end{equation}
for certain constants $A$ and $B$, depending on $n$. The only property of these constants that we will be using is that
\begin{equation}
\label{eq:structureass}
A^2>4B>0 
\qquad\text{and}\qquad
A>0 \,.
\end{equation}

In \cite{FrKo} we have classified all entire solutions (i.e., solutions defined on all of $\R$) of equation \eqref{ode with g} in the case $g(v) = v^\frac{n+4}{n-4}$. In the present section we extend this classification to all $g$ satisfying the assumptions \eqref{conditions on g}. In a complementary way to the proof of Theorem \ref{theorem moving planes}, in this section we actually make no use of the upper bound $g(t)/t \leq (n+4)/(n-4) g'(t)$ from \eqref{conditions on g}. 

Our exposition here will  focus on the main steps of the argument, providing details only where there is a significant difference to \cite{FrKo}.


\subsection{Proof of Theorem \ref{theorem classification ODE}}

We will assume throughout this section that $g$ satisfies the assumptions \eqref{conditions on g}. Let us introduce
\begin{equation}
\label{definition G}
 G(v) = \int_0^v g(t) \diff t  - B \frac{v^2}{2}. 
\end{equation}
The crucial observation is that for the proofs in \cite{FrKo} only some qualitative properties of $G$ are needed, as described in the next lemma. 

\begin{lemma}
\label{lemma properties of G}
One has $G(0) = G'(0) = 0$.  Moreover, there is an $a_0 >0$ such that $G$ is strictly decreasing on $(0, a_0)$ and strictly increasing towards $\infty$ on $(a_0, \infty)$.
\end{lemma}

\begin{proof}
The first part follows directly from \eqref{definition G} and the fact that $\lim_{t \to 0} g(t) = 0$. 

For the second part, we write $G'(v) = g(v) - B v = v (\frac{g(v)}{v} - B)$. The assumptions \eqref{conditions on g} imply that $\frac{g(v)}{v}-B$ is negative for all sufficiently small $v$ and that $v \mapsto \frac{g(v)}{v}$ is strictly increasing. Thus, there is an $a_0 > 0$ such that $G'$ is strictly negative on $(0, a_0)$ and strictly positive on $(a_0, \infty)$. From the assumption $g(t) \geq c t^q$ in \eqref{conditions on g} we obtain $G(v) \geq \frac{c}{q+1} v^{q+1} - \frac{B}2 v^2 - C \to \infty$ as $v \to \infty$.
\end{proof}

The following lemma concerns the asymptotic behavior of entire solutions. Its most notable consequence is that any solution to \eqref{ode with g} which tends to $\infty$ must blow up in finite time. 

\begin{lemma}
\label{lemma wei blow up}
Let $v \in C^4((0, \infty))$ be a positive solution of \eqref{ode with g} and suppose that $a:= \lim_{t \to \infty} v(t) \in [0, \infty]$ exists. Then either $a = 0$ or $a = a_0$. 
\end{lemma}

Lemma \ref{lemma wei blow up} is proved in \cite[Proposition 5]{GazzolaGrunau} for $v \in C^4(\R)$ and in the special case $g(t) = t^p$ with $p > \frac{n+4}{n-4}$. An inspection of the proof there shows that under the assumptions of Lemma \ref{lemma wei blow up} the same method of proof can be applied, using Lemma \ref{lemma properties of G} and the inequality $g(t) \geq c t^q$ from \eqref{conditions on g}. We omit the details. 

The following comparison lemma is a key technical ingredient in our argument. It is a variant of \cite[Lemma 1]{vdb}, see also \cite[Theorem 2.1]{BuChTo}. The novelty here is that it is stated and proved for nonnegative solutions, instead of bounded solutions. This difference allows to significantly shorten the proof in \cite{FrKo}, because boundedness of entire solutions, which was one of the main steps in \cite{FrKo}, need no longer be shown a priori.

For the statement of the lemma we recall the structural assumption \eqref{eq:structureass}, which implies that the polynomial $\xi^2 - A \xi + B$ has two distinct positive roots. We denote these by $\lambda > \mu>0$.

\begin{lemma}
\label{lemma comparison}
Let $v, w \in C^4(\R)$ be nonnegative solutions to the equation \eqref{ode with g} with 
\begin{align*}
v(0) & \geq w(0), \\
  v'(0) &\geq w'(0), \\
v''(0) - \mu v(0) &\geq w''(0) - \mu w(0), \\
 v'''(0) - \mu v'(0) &\geq w'''(0) - \mu w'(0).
\end{align*}
Then $v \equiv w$.
\end{lemma}

Because of its importance for us, we include a complete proof of this lemma. It follows closely that of \cite[Lemma 1]{vdb}, but uses in addition Lemma \ref{lemma wei blow up}.

\begin{proof}
Let $v$ and $w$ satisfy the assumptions of the lemma and suppose, by contradiction, that $v \nequiv w$. Then by uniqueness of ODE solutions and by our hypotheses on the initial conditions, there is $k \in \{0,1,2,3\}$ such that $v^{(k)}(0) > w^{(k)}(0)$ and $v^{(l)}(0) = w^{(l)}(0)$ for all $0 \leq l < k$. Therefore, in any case, 
\begin{equation}
\label{eq:comparison}
v > w \quad \text{on } (0, \sigma)
\end{equation}
for some sufficiently small $\sigma >0$. 

We define the auxiliary functions 
\[ \phi(t):= v''(t) - \mu v(t) \qquad \text{and}\qquad \psi(t):= w''(t) - \mu w(t). \]
Then by the hypotheses, we have
\begin{equation} \label{eq comp 1} (\phi-\psi)(0) \geq 0 \qquad \text{and}\qquad (\phi-\psi)'(0) \geq 0. \end{equation}
From equation \eqref{ode with g} and by the definition of $\lambda$ and $\mu$, we have
\[ (\phi-\psi)''(t) - \lambda (\phi-\psi)(t) = g(v(t)) - g(w(t)) 
\quad \text{ for all   } t \in \R.
\]
Because of \eqref{eq:comparison} and the fact that $g$ is strictly increasing on $(0, \infty)$, this implies that
\begin{equation} \label{eq comp 2} (\phi-\psi)''(t) - \lambda (\phi-\psi)(t) > 0 \quad \text{ for all   } t \in (0, \sigma). \end{equation}
The inequalities \eqref{eq comp 1} and \eqref{eq comp 2} easily imply that $(\phi - \psi)(t) \geq 0$ for $t \in (0, \sigma)$, or equivalently, that
\begin{equation} \label{eq comp 3} (v-w)''(t) \geq \mu (v-w) (t)  \quad \text{ for all   } t \in (0, \sigma). \end{equation}
Since $(v-w)'(0) \geq 0$ by the hypotheses of the lemma, we see from \eqref{eq comp 3} and \eqref{eq:comparison} that $(v-w)'(t) >0$ for all $t \in (0, \sigma)$. Hence $v-w$ is strictly increasing on $(0, \sigma)$ and, since $\sigma >0$ was arbitrary with the property \eqref{eq:comparison}, we infer that $v-w$ remains strictly positive for all times. 

Repeating the above arguments for the interval $(0, \infty)$ instead of $(0, \sigma)$, we see from \eqref{eq comp 3} and \eqref{eq:comparison} that $(v-w)'$ is positive and strictly increasing on $(0, \infty)$. Thus $\lim_{ t \to \infty} (v-w)(t) = \infty$. Since $w$ is nonnegative, this implies, in particular, that $\lim_{t \to \infty} v(t) = \infty$. This contradicts Lemma \ref{lemma wei blow up}, and we have therefore proved that $v \equiv w$. 
\end{proof}

From Lemma \ref{lemma comparison} we can deduce a remarkable rigidity property, namely that positive entire solutions to \eqref{ode with g} are determined by only two (instead of four!) initial values. A simple consequence of this is that positive solutions of \eqref{ode with g} are symmetric with respect to local extrema.   

\begin{corollary}
\label{corollary symmetry}
\begin{enumerate}
\item[(i)] Let $v, w \in C^4(\R)$ be nonnegative solutions of \eqref{ode with g} with $v(0) = w(0)$ and $v'(0) = w'(0)$. Then $v \equiv w$.
\item[(ii)]
Suppose that $v \in C^4(\R)$ is a nonnegative solution of \eqref{ode with g} with $v'(t_0) = 0$ for some $t_0 \in \R$. Then $v$ is symmetric with respect to $t_0$, i.e. $v(t_0 + t) = v(t_0 - t)$ for all $t \in \R$. 
\end{enumerate}
\end{corollary}

We point out once more that we only assume nonnegativity of $v$ in Corollary \ref{corollary symmetry}, whereas in \cite[Corollary 5]{FrKo} we assumed boundedness of $v$.  

\begin{proof}
[Proof of Corollary \ref{corollary symmetry}]
To prove $(i)$, we observe that up to exchanging $v$ and $w$, we may assume $v''(0) \geq w''(0)$. Furthermore, up to replacing $v(t)$ and $w(t)$ by $v(-t)$ and $w(-t)$ (which still solve \eqref{ode with g}), we may assume $v'''(0) \geq w'''(0)$. Then all assumptions of Lemma \ref{lemma comparison} are satisfied and we conclude $v \equiv w$. 

To prove $(ii)$, we simply apply $(i)$ to $v$ and $w(t) = v(t_0 - t)$, which also solves \eqref{ode with g}.
\end{proof}

We now prove a variant of Lemma \ref{lemma comparison} where one of the functions is constant.

\begin{lemma}\label{compconst}
Let $v \in C^4(\R)$ be a positive solution of \eqref{ode with g} and assume that either
\begin{align*}
v(0) \geq a_0, \qquad
  v'(0)= 0, \qquad
v''(0) \geq 0, \qquad
 v'''(0)= 0
\end{align*}
or
\begin{align*}
v(0) \leq a_0, \qquad
  v'(0)= 0, \qquad
v''(0) \leq 0, \qquad
 v'''(0)= 0 \,.
\end{align*}
Then $v \equiv a_0$. 
\end{lemma}

\begin{proof} \textit{Proof under the first set of assumptions.}
Suppose, by contradiction, that $v \nequiv a_0$. Then by uniqueness of ODE solutions, either $v(0) > a_0$ or $v''(0) > 0$. 
Moreover, from the equation, we have
\begin{equation}
\label{eq for fourth derivative}
v^{(4)}(t) = A v''(t) + \big( g(v(t)) - Bv(t) \big).  
\end{equation}
Observing that $g(v) - Bv > 0$ for $v \in (a_0, \infty)$, we deduce that in both cases ($v(0) > a_0$ or $v''(0) > 0$) we have $v^{(4)} > 0$ on $(0, \sigma)$ for some sufficiently small $\sigma >0$. 

Together with the initial conditions, this implies that $v^{(k)}$ is strictly increasing on $(0, \sigma)$ for $k = 0,1,2,3$. Since $\sigma >0$ was arbitrary with the property that the right side of \eqref{eq for fourth derivative} is positive, we infer, in particular, that $v'$ is positive and strictly increasing on $(0, \infty)$. Thus $\lim_{ t \to \infty} v(t) = \infty$. This contradicts Lemma \ref{lemma wei blow up}, and we have therefore proved that $v \equiv a_0$. 

\textit{Proof under the second set of assumptions.}
Suppose, by contradiction, that $v \nequiv a_0$. Then by uniqueness of ODE solutions, either $v(0) < a_0$ or $v''(0) < 0$. Observing that $g(v) - Bv < 0$ for $v \in (0, a_0)$, we deduce from \eqref{eq for fourth derivative} that in both cases ($v(0) < a_0$ or $v''(0) < 0$) we have $v^{(4)} < 0$ on $(0, \sigma)$ for some sufficiently small $\sigma >0$. 

Together with the initial conditions, this implies that $v^{(k)}$ is strictly decreasing on $(0, \sigma)$ for $k = 0,1,2,3$. Since $\sigma >0$ was arbitrary with the property that the right side of \eqref{eq for fourth derivative} is negative, we infer, in particular, that $v'$ is negative and strictly decreasing as long as $v > 0$. Therefore, we must have $v(t_0) = 0$ with $v'(t_0) < 0$ for some finite $t_0 < \infty$.  This contradicts the positivity of $v$ and we have therefore proved that $v \equiv a_0$. 
\end{proof}

%

At last, we give the main ideas for the proofs of Theorems \ref{theorem classification ODE} and \ref{theorem classification} using the ingredients introduced so far.

\begin{proof}
[Proof sketch of Theorem \ref{theorem classification ODE}]
By arguments detailed in \cite[Proof of Proposition 3]{FrKo}, we can deduce from Lemmas \ref{lemma wei blow up} and \ref{lemma comparison} that every positive solution $v \in C^4(\R)$ is either constant equal to $a_0$, homoclinic to zero or periodic with unique local maximum and minimum per period.

We can now prove $(i)$. The existence and uniqueness of $a_0$ with the claimed properties is contained in Lemma \ref{lemma properties of G}. Next, let $v\in C^4(\R)$ be a positive solution with $\min_\R v \geq a_0$. Suppose without loss that $v(0) = \min_\R v$. Then $v'(0)=0$, $v''(0) \geq 0$ and, by Corollary~\ref{corollary symmetry}, $v$ is symmetric, so $v'''(0)=0$. Thus, Lemma \ref{compconst} implies $v \equiv a_0$. This completes the proof of $(i)$. 

The existence part of assertion $(ii)$ can be proved via the shooting method as in \cite{FrKo}. We invite the reader to check that using the properties of $G$ stated in Lemma~\ref{lemma comparison}, the argument carries over to the more general case considered here. For the uniqueness part of assertion $(ii)$ we use the fact mentioned above that any solution $v$ with $\inf_\R v=a>0$ is either constant or periodic (since it cannot be homoclinic to zero). In particular, $\inf_\R v$ is attained. Therefore the uniqueness follows from the first part of Corollary \ref{corollary symmetry}. The stated periodicity and monotonicity properties follow from the second part of Corollary \ref{corollary symmetry}.

Finally, we prove $(iii)$. We obtain the existence of a homoclinic solution simply as a limit of periodic solutions. Indeed, if we denote by $v_a$ the periodic solutions obtained in $(ii)$ with the normalization $v_a(0)=\max v_a$ and by $L_a$ their period length, then $v_a$ is symmetric-decreasing on $(-\frac{L_a}{2}, \frac{L_a}{2})$. Using $L_a \to \infty$ as $a \to 0$, it is not difficult to prove that $v_a$ converges to a non-trivial limit function which must be the homoclinic solution. See Subsection \ref{subsection existence of homoclinic solution} for details.

Next, we prove the uniqueness claim in $(iii)$. We first note that if $v$ is a positive solution with $\inf_\R v=0$, then $v$ is homoclinic to zero. This follows from the fact mentioned above, since if $v$ periodic or constant, it cannot be positive and have $\inf_\R v=0$. Now let $v$ and $w$ be two positive solutions in $C^4(\R)$ with the property that $v'(0)=w'(0)=0$ and $\lim_{|t|\to\infty} v(t)=\lim_{|t|\to\infty} w(t)=0$. We argue by contradiction and assume $v \nequiv w$. Then, by the first part of Corollary \ref{corollary symmetry}, we may assume that $v(0) > w(0)$. By comparison arguments detailed in \cite[Lemma~9]{FrKo}, this enforces that $v(t) > w(t)$ for all $t \in \R$. 

We can now derive the desired contradiction. For every $R >0$, we have, using integration by parts and the fact that $v$ and $w$ satisfy \eqref{ode with g},
\begin{align*}
0 & = \int_{-R}^R w(v^{(4)} - A v'' + Bv - g(v)) \\
&= b(R) + \int_{-R}^R v(w^{(4)} - A w'' +Bw -g(w)) + \int_{-R}^R wv(\frac{g(w)}{w} - \frac{g(v)}{v}) \\
&= b(R) + \int_{-R}^R wv(\frac{g(w)}{w} - \frac{g(v)}{v}) \,.
\end{align*}
Here, $b(R)$ contains all the boundary terms coming from the integrations by part. As in \cite[Lemma 4]{vdb} one shows that $b(R) \to 0$ as $R \to \infty$. But since the function $t \mapsto \frac{g(t)}{t}$ is strictly increasing on $(0, \infty)$ and since $v > w$, we find that $\int_{-R}^R wv(\frac{g(w)}{w} - \frac{g(v)}{v})$ is a negative and strictly decreasing function of $R$. Thus we obtain a contradiction by choosing $R$ large enough.

Finally, the claimed decay behavior can be proved by relatively standard comparison arguments, again relying on the factorization structure of equation \eqref{ode with g}; see Subsection~\ref{subsection decay homoclinic} for details. This concludes the proof of Theorem \ref{theorem classification ODE}.
\end{proof}

\begin{proof}[Proof of Remark \ref{remark strict}]
The inequality $\frac{\partial u}{\partial |x|}<0$ is equivalent to the bound $v'<\frac{n-4}{2}v$. Similarly as in the proof of Lemma \ref{lemma comparison} we introduce $\mu$ and $\lambda$ and
\[ \phi(t):= v''(t) - \mu v(t) \,,\]
which satisfies
\begin{equation}
\label{eq:factorization}
\phi'' -\lambda \phi = g(v) \,.
\end{equation}
(We emphasize that here $\mu=(n-4)^2/4$, which is potentially different from the use of $\mu$ in Theorem \ref{theorem classification ODE}.) Since $g(v)>0$, it follows from the maximum principle that $\phi<0$ on $\R$. Indeed, by Theorem \ref{theorem classification ODE} we know that $v$ is either constant, periodic or homoclinic to zero. In the first two cases the maximum principle can be clearly applied. In the third case it can be applied since $\lim_{|t|\to\infty} \phi(t)=0$, as we already argued in the proof of Theorem \ref{theorem classification ODE}.

The function $w:= v'/v$ satisfies
\begin{equation}
\label{eq:radmono}
w' = -w^2 + \mu + \frac{\phi}{v} \,.
\end{equation}
Since $\phi < 0$ and $v>0$, we have 
\begin{equation}
\label{ineq for w}
w' < - w^2 + \mu
\end{equation}
In particular, $w'(t) < 0$ whenever $|w(t)| \geq \sqrt \mu$. This implies that the set $\{ w \geq \sqrt \mu \}$ is either empty, or equal to $\R$ or of the form $(- \infty, T]$ for some $T \in\R$. We will rule out the last two possibilities and conclude that $w<\sqrt\mu$, as claimed.

We make again use of the classification result from Theorem \ref{theorem classification ODE}. When $v$ is constant or periodic, there is a two-sided sequence $(t_n)_{n\in\Z}$ with $t_n\to\pm\infty$ as $n\to\pm\infty$ such that $v'(t_n)=0$ for all $n$. This implies $w(t_n)=0$ and therefore the set in question can neither be of the form $\R$ nor of the form $(-\infty,T]$ for $T\in\R$.

Now assume that $v$ is homoclinic to zero. Then there is a $t_0\in\R$ such that $v'(t_0)=0$ and therefore the set cannot be of the form $\R$. Now suppose that there is a $T\in\R$ such that $w(t) \geq \sqrt \mu$ for all $t \leq T$. By \eqref{ineq for w}, $w' < 0$ on $(-\infty, T]$. Therefore there is an $\epsilon > 0$ such that $w(t) \geq \sqrt{\mu + \epsilon}$ for all $t \leq T-1 =: T_1$. Thus, $\mu\leq (\mu/(\mu+\epsilon))w(t)^2$ and so \eqref{ineq for w} implies $w'< - (\epsilon/(\mu+\epsilon)) w^2$ on $(-\infty,T_1)$ and therefore by integration,
$$
\frac{1}{w(T_1)} - \frac{1}{w(t)} > \frac\epsilon{\mu+\epsilon} (T_1-t)
\qquad\text{for all}\ t<T_1 \,.
$$
Since $w(t)>0$, this is a contradiction for $t$ sufficiently negative. This completes the proof.
\end{proof}


\subsection{Existence of a homoclinic solution}
\label{subsection existence of homoclinic solution}

In this subsection we provide the details in the existence part of $(iii)$ in Theorem \ref{theorem classification ODE}. As already explained, we shall construct a homoclinic solution to \eqref{ode with g} as a limit of periodic solutions $v_a$ with $a \searrow 0$.

The family $(v_a)_{a \in (0, a_0)}$ of periodic solutions and their associated minimal period lengths $(L_a)_{a \in (0, a_0)}$ were introduced before the statement of Theorem \ref{theorem classification}. We recall that $v_a(0) = \max_\R v_a$, that $v_a( \pm \frac{L_a}{2}) = \min_\R v_a = a$ and that $v_a$ is strictly symmetric-decreasing on $(- \frac{L_a}{2}, \frac{L_a}{2})$. 

The following lemma is fundamental for our construction. 

\begin{lemma}
\label{lemma uniform convergence}
Let $a_n \searrow 0$ and consider $v_n = v_{a_n}$ with associated minimal period length $L_n = L_{a_n}$. Then there is $v_\infty \in C^4(\R)$ such that, up to extracting a subsequence, we have $v_0 \to v_\infty$ in $C^4(K)$ for every compact $K \subset \R$. Moreover, $v_\infty(0) \geq a_0$ and, in particular, $v_\infty \nequiv 0$. 
\end{lemma}

\begin{proof}
Fix $R > 0$ and consider the compact interval $[-R, R]$. We shall prove that 
\begin{equation}
\label{uniform bound on v and v''}
 \sup_{n \in \N} \sup_{t \in [-R, R]} \big( |v_n(t)|+|v''_n(t)|  \big) < \infty.
\end{equation}
Indeed, if \eqref{uniform bound on v and v''} holds, then by equation \eqref{ode with g}, $v_n^{(4)} = Av''_n - Bv_n + g(v_n)$ is also bounded on $[-R, R]$ and so are $v'_n(t) = \int_0^t v''_n(s)\diff s$ and $v'''_n(t) = \int_0^t v^{(4)}_n(s)\diff s$. By Arzelà-Ascoli, up to a subsequence, we thus have $v_n \to v_\infty$ in $C^3([-R, R])$ and, by using the equation again, in $C^4([-R,R])$. Since $R >0$ was arbitrary, we conclude by a diagonal argument.

To prove \eqref{uniform bound on v and v''}, we consider the energy
\[ \mathcal E_v(t) = -v'(t)v'''(t) + \frac{v''(t)^2}{2} + A \frac{v'(t)^2}{2} + G(v(t)). \]
By differentiating and using \eqref{ode with g} we see that $\mathcal E_v(t)$ is constant with respect to $t$. Moreover, by \cite[Corollary 6]{vdb}, we have
\[ \frac{v_n''(t)^2}{2} + G(v_n(t)) \leq \mathcal E_{v_n} \]
for all $n\in \N$. Since $G(v) \to \infty$ as $v \to \infty$, boundedness of both $v''_n(t)$ and $v_n(t)$ will follow if we can prove that $\mathcal E_{v_n}$ is bounded uniformly in $n$.

To do so, we claim that $v_n''(\frac{L_{a_n}}{2}) \leq b/A$, where $b:= \max_{v \in [0, \infty)} (Bv - g(v))$. (We recall from the proof of Lemma \ref{lemma properties of G} that $0<b<\infty$.) Indeed, if this bound on $v_n''(\frac{L_a}{2})$ was not true, the initial conditions and the equation 
\[ v_n^{(4)} = Av_n'' -Bv_n + g(v_n) \]
would imply that $v_n^{(4)}$ is positive, and $v_n''$ is increasing for all times. However, this is impossible because $v_n$ is periodic. 

Thus, since $v_n(\frac{L_{a_n}}{2}) = a_n \in (0, a_0)$ and therefore $G(v_n(\frac{L_{a_n}}{2})) < 0$, we have
$$
\mathcal E_{v_n} = \mathcal E_{v_n}(\tfrac{L_{a_n}}{2}) = \frac{v_n''(\frac{L_{a_n}}{2})^2}{2} + G(v_n(\tfrac{L_{a_n}}{2})) \leq \frac{b^2}{2A^2}
$$
for all $n \in \N$. This finishes the proof of \eqref{uniform bound on v and v''}.

Lastly, we prove that $v_\infty(0) \geq a_0$. We first note that the inequality $v_a(0) > a_0$ follows from Lemma \ref{compconst}, similarly as in the proof of $(i)$ in Theorem \ref{theorem classification ODE}. Letting $a \to 0$, we find $v_\infty(0) \geq a_0$.
\end{proof}

The second observation that we need is that the period length diverges as the minimum value approaches 0. 

\begin{lemma}
\label{lemma period length to infty}
As $a \to 0$, $L_a \to \infty$. 
\end{lemma}

\begin{proof}
Suppose that there is a sequence $a_n \searrow 0$ and $L_\infty < \infty$ such that $L_n := L_{a_n} \to L_\infty$. Then by Lemma \ref{lemma uniform convergence}, up to a subsequence, there is a nonnegative $v_\infty \in C^4(\R)$ which solves \eqref{ode with g} such that $v_n := v_{a_n} \to v_\infty$ in $C^4([- L_\infty,  L_\infty])$. Moreover, we have $v_\infty(\frac{L_\infty}{2}) = \lim_{n \to \infty} v_n(\frac{L_n}{2}) = \lim_{n \to \infty} a_n = 0$ and $v'_\infty(\frac{L_\infty}{2}) = \lim_{n \to \infty} v'_n(\frac{L_n}{2}) = 0$. From Corollary \ref{corollary symmetry} with $w\equiv 0$ we deduce that $v_\infty \equiv 0$, in contradiction to Lemma \ref{lemma uniform convergence}.
\end{proof}

We can now prove the desired existence result. 

\begin{lemma}
There is a positive solution $v_0 \in C^4(\R)$ of \eqref{ode with g} with $\lim_{|t| \to\infty} v_0(t) = 0$. 
\end{lemma}

\begin{proof}
By Lemma \ref{lemma uniform convergence}, there is a nonnegative solution $v_0 \in C^4(\R)$ of \eqref{ode with g} such that $v_n \to v_0$ in $C^4(K)$ for every compact $K \subset \R$. Since $v_n$ is symmetric-decreasing on $[-\frac{L_n}{2}, \frac{L_n}{2}]$, and since $L_n \to \infty$ by Lemma \ref{lemma period length to infty}, $v_0$ is symmetric-decreasing on all of $\R$ and therefore $\lim_{t \to \infty} v_0(t)$ exists. By Lemma \ref{lemma wei blow up}, this limit equals either $0$ or $a_0$ and it remains to exclude the second case. 

Thus, suppose that $\lim_{t \to \infty} v_0(t) = a_0$. We can derive a contradiction using the energy $\mathcal E_v$ introduced in the proof of Lemma \ref{lemma uniform convergence}. Using the fact that all derivatives of $v_0$ vanish at $\infty$ by \cite[Lemma 4]{vdb} (note that $v_0$ is monotone as required for this lemma), we have
\begin{equation}
\label{eq:homoexcont1}
\mathcal E_{v_0} = \lim_{t \to \infty} \mathcal E_{v_0} (t) = G(a_0) < 0 \,. 
\end{equation}
On the other hand, we have for each $n$,
$$
\mathcal E_{v_n} = \mathcal E_{v_n}(\tfrac{L_{a_n}}{2}) = \frac{v_n''(\frac{L_{a_n}}{2})^2}{2} + G(v_n(\tfrac{L_{a_n}}{2})) = \frac{v_n''(\frac{L_{a_n}}{2})^2}{2} + G(a_n) \geq G(a_n) \,.
$$
Thus, since $v_n\to v_0$ in $C^4(K)$ for any compact $K$ implies $\mathcal E_{v_n} \to \mathcal E_{v_0}$ as $n\to\infty$ and since $G(a_n)\to G(0)=0$ as $n\to\infty$, we have
$$
\mathcal E_{v_0} = \lim_{n\to\infty} \mathcal E_{v_n} \geq \lim_{n\to\infty} G(a_n)=0 \,,
$$
contradicting \eqref{eq:homoexcont1}. This contradiction shows that $\lim_{t \to \infty} v_0(t) = 0$. By the symmetry of $v_0$, we also obtain $\lim_{t \to - \infty} v_0(t) = 0$ and the proof is complete. 
\end{proof}


\subsection{Decay behavior of the homoclinic solution}
\label{subsection decay homoclinic}

We prove the following decay behavior of the homoclinic solution of \eqref{ode with g}. We recall that we assume $0\leq \beta=\lim_{s \to 0} g'(s) < B$ and set $\mu =\frac12(A - \sqrt{A^2 - 4(B - \beta)})$.

\begin{lemma}
\label{lemma decay of homoclinic solution}
Let $v \in C^4(\R)$ be a positive solution of \eqref{ode with g} with $\lim_{|t| \to \infty} v(t) = 0$. Then for every $\epsilon>0$ there is a $C_\epsilon<\infty$ such that $v(t) \leq C_\epsilon e^{-(\sqrt\mu-\epsilon)|t|}$ for all $t\in\R$. 
\end{lemma}

The proof of this bound relies on a comparison argument using the following fourth-order variant of the maximum principle. 

\begin{lemma}
\label{lemma max principle fourth order}
Suppose that $w \in C^4(\R)$ satisfies the inequality 
\[ (\frac{\diff^2}{\diff t^2} - \lambda) (\frac{\diff^2}{\diff t^2} - \mu) w(t) \geq 0 \quad \text{ on } (T, \infty) \]
for some $\lambda, \mu > 0$ and $T \in \R$ and that $w(T)= w''(T)  = \lim_{t \to \infty} w(t) =\lim_{t \to \infty} w''(t)= 0$. Then $w \geq 0$ on $(T, \infty)$. 
\end{lemma}

\begin{proof}
The inequality for $w$ factorizes into the system 
\[ 
\begin{cases}
w'' - \lambda w = u, \\
u'' - \mu u \geq 0. 
\end{cases}
\]
By assumption, we have $u(T) = 0$ and $\lim_{t \to \infty} u(t) = 0$. By the maximum principle applied to the inequality $u''-\mu u\geq 0$, we thus deduce that 
\[ w'' - \lambda w = u \leq  0. \]
By assumption, we have $w(T) = 0$ and $\lim_{t \to \infty} w(t) = 0$. Applying the maximum principle a second time, we deduce $w \geq 0$, as desired. 
\end{proof}

\begin{proof}
[Proof of Lemma \ref{lemma decay of homoclinic solution}]
We first observe that the limiting behavior of $g$ and $g'$ at 0 given by \eqref{conditions on g}, together with the mean value theorem easily imply that $\lim_{v \to 0} \frac{g(v)}{v}$ exists and is equal to $\beta = \lim_{v \to 0} g'(v)$.  
Therefore, we may write $g(v) = \beta v + h(v)$ with $\lim_{v \to 0} \frac{h(v)}{v} = 0$ and think of equation \eqref{ode with g} as 
\begin{equation}
\label{ode with h} 
v^{(4)} - Av'' + (B-\beta)v = h(v). 
\end{equation}
Fix some $\epsilon > 0$ such that $B - \beta - \epsilon > 0$ and $T_\epsilon$ such that $\frac{h(v(t))}{v(t)} \leq \epsilon$ for all $t > T_\epsilon$. Since $v > 0$, equation \eqref{ode with h} implies
\[ v^{(4)} - A v '' + (B- \beta - \epsilon) v \leq 0 \quad \text{ on } (T_\epsilon, \infty). \]
Since by our choice of $\epsilon$, the condition $A^2 > 4 (B- \beta - \epsilon) > 0$ is still satisfied, the expression on the left hand side factorizes as 
\[ (\frac{\diff^2}{\diff t^2} - \lambda_\epsilon) (\frac{\diff^2}{\diff t^2} - \mu_\epsilon) v(t) \leq 0 \quad \text{ on } (T_\epsilon, \infty), \]
with some $\lambda_\epsilon > \mu_\epsilon > 0$. 

We compare $v$ with a solution $f_\epsilon$ of
\[  (\frac{\diff^2}{\diff t^2} - \lambda_\epsilon) (\frac{\diff^2}{\diff t^2} - \mu_\epsilon) f_\epsilon = 0 \]
which tends to zero as $t \to \infty$. The general solution of this problem is given by 
\[ f_\epsilon(t) = a_\epsilon e^{-\sqrt{\lambda_\epsilon} t} + b_\epsilon e^{-\sqrt{\mu_\epsilon}t }, \qquad a_\epsilon, b_\epsilon \in \R\, .  \]
We may fix $a_\epsilon$ and $b_\epsilon$ such that 
\[ f_\epsilon(T_\epsilon) = v(T_\epsilon) \quad \text{and} \quad f''_\epsilon(T_\epsilon) = v''(T_\epsilon) \,. \]
(This is always possible as a consequence of $\lambda_\epsilon \neq \mu_\epsilon$). Since moreover, $\lim_{t \to \infty} v''(t) =0$ by \cite[Lemma 4]{vdb}, the function $w := f_\epsilon - v$ fulfills the assumptions of Lemma \ref{lemma max principle fourth order}. We therefore deduce that $w \geq 0$, i.e., $v \leq f_\epsilon = O(e^{- \sqrt{\mu_\epsilon} t})$ as $t \to \infty$. Analogously, one obtains $v = O(e^{ \sqrt{\mu_\epsilon} t})$ as $t \to - \infty$. Since $\mu_\epsilon\to \mu$ as $\epsilon\to 0$, we obtain the claimed bound. 
\end{proof}

\begin{lemma}
\label{lemma decay of homoclinic solution sharp}
Let $v \in C^4(\R)$ be a positive solution of \eqref{ode with g} with $\lim_{|t| \to \infty} v(t) = 0$ and assume that $g$ satisfies the additional assumption \eqref{eq:gnearzero}. Then $\lim_{|t|\to\infty} e^{\sqrt\mu |t|}v(t)$ exists and is finite. When $g(t)\geq \beta t$ for all $t>0$, then the limit is positive.
\end{lemma}

\begin{proof}
Equation \eqref{ode with h} reads, in factorized form, 
\begin{equation}
\label{ode factorized} (\frac{\diff^2}{\diff t^2} - \lambda)(\frac{\diff^2}{\diff t^2} - \mu) v = h(v), \end{equation}
with $\lambda> \mu > 0$ given by
\begin{equation}
\label{lambda and mu}
\lambda = \frac{A + \sqrt{A^2 - 4(B - \beta)}}{2} \quad \text{  and  } \quad \mu =\frac{A - \sqrt{A^2 - 4(B - \beta)}}{2}.
\end{equation}
The Green's function associated to equation \eqref{ode factorized} is 
\[ G(t,s) = \frac{1}{\lambda-\mu} \left( \frac{e^{-\sqrt\mu |t-s|}}{2\sqrt\mu} - \frac{e^{-\sqrt\lambda|t-s|}}{2\sqrt\lambda} \right).
\]
According to \eqref{eq:gnearzero} and Lemma \ref{lemma decay of homoclinic solution} we have
\begin{equation}
\label{eq:hvbound}
|h(v(t))| \leq C_\epsilon e^{-r(\sqrt{\mu}-\epsilon)|t|}
\end{equation}
As a consequence, we can solve equation \eqref{ode factorized} for $v$ using $G$ and obtain
\[ v(t) = \int_\R G(t,s) h(v(s)) \,. \]
Using this formula and \eqref{eq:hvbound} with $\epsilon>0$ chosen so small that $r(\sqrt\mu-\epsilon)>\sqrt\mu$ it is easy to deduce that
$$
\lim_{t\to\infty} e^{\sqrt\mu t} v(t) = \frac{1}{\lambda-\mu}\, \frac{1}{2\sqrt\mu} \int_\R e^{\sqrt\mu s} h(v(s)) \,,
$$
where the right side is finite. Moreover, if $h\geq 0$, then the limit is positive. (Note that \eqref{conditions on g} implies that $h(v)/v<h'(v)$ for all $v>0$ and therefore $h$ is non-zero.) 
\end{proof}

\bibliography{paper-mazya}
	\bibliographystyle{plain}

\end{document}